\documentclass[12pt]{amsart}
\usepackage{graphicx}
\usepackage{amscd}
\usepackage[top=1in,bottom=1in,left=1.15in,right=1.15in]{geometry}
\usepackage{amsmath,amssymb,amsfonts}
\usepackage{amsthm}
\usepackage{color}
\usepackage{hyperref}
\usepackage{tikz-cd}
\usetikzlibrary{plotmarks}
\usepackage{epstopdf}
\usetikzlibrary{matrix}
\usepackage{verbatim}
\usepackage{braket}
\usepackage{mathrsfs}
\newtheorem{theo}{Theorem}[section]
\newtheorem{prop}[theo]{Proposition}
\newtheorem{lemma}[theo]{Lemma}
\newtheorem{defn}[theo]{Definition}

\newtheorem{ex}[theo]{Example}

\allowdisplaybreaks

\usetikzlibrary{decorations.markings}

\tikzset{->-/.style={decoration={
  markings,
  mark=at position .5 with {\arrow{>}}},postaction={decorate}}}

\tikzset{-<-/.style={decoration={
  markings,
  mark=at position .5 with {\arrow{<}}},postaction={decorate}}}

\makeatletter
\@namedef{subjclassname@2020}{\textup{2020} Mathematics Subject Classification}
\makeatother

\begin{document}

\title[$\mathfrak{gl}(1 \vert 1)$-Alexander polynomial for $3$-manifolds]{$\mathfrak{gl}(1 \vert 1)$-Alexander polynomial for $3$-manifolds}
\author{Yuanyuan Bao and Noboru Ito}

\address{
Graduate School of Mathematical Sciences, University of Tokyo, 3-8-1 Komaba, Tokyo 153-8914, Japan
}
\email{bao@ms.u-tokyo.ac.jp}

\address{National Institute of Technology, Ibaraki College, 866 Nakane, Hitachinaka, Ibaraki, 312-8508, Japan
}
\email{nito@ibaraki-ct.ac.jp}


\maketitle

\begin{abstract}
As an extension of Reshetikhin and Turaev's invariant, Costantino, Geer and Patureau-Mirand constructed $3$-manifold invariants in the setting of relative $G$-modular categories, which include both semisimple and non-semisimple ribbon tensor categories as examples. In this paper, we follow their method to construct a $3$-manifold invariant from Viro's $\mathfrak{gl}(1\vert 1)$-Alexander polynomial. 
We take lens spaces $L(7, 1)$ and $L(7, 2)$ as examples to show that this invariant can distinguish homotopy equivalent manifolds.
\end{abstract}

\section{Introduction}

A $3$-manifold in this paper indicates a connected compact closed oriented smooth $3$-manifold. Given a framed link $L$ in $S^3$, the integral surgery along $L$ produces a $3$-manifold. The link $L$ is called a surgery presentation of the resulting manifold. Kirby calculus \cite{MR467753} says that any $3$-manifold can be obtained in this way. In addition, surgery presentations of the same $3$-manifold are related to each other by Kirby moves. 

A linear sum of quantum invariants of framed links defines a topological invariant for $3$-manifolds, if it is invariant under Kirby moves. Reshetikhin and Turaev \cite{MR1091619} gave the first rigorous construction of $3$-manifold invariant along this line. Their invariant was defined for a modular category, which is semisimple and all simple objects are required to have non-zero quantum dimensions. 

Costantino, Geer and Patureau-Mirand \cite{MR3286896} extended Reshetikhin and Turaev's construction to categories which may not be semisimple or may contain objects with zero quantum dimensions. They proposed the concept: relative $G$-modular category and proved that the quantum invariant of framed links constructed from a relative $G$-modular category can be used to define a $3$-manifold invariant. Let $\mathscr{C}$ be a relative $G$-modular category. For a $3$-manifold $M$, a $\mathscr{C}$-ribbon graph $T$ and a cohomology class $\omega: H_1(M\backslash T, \mathbb{Z})\to G$ which satisfy some compatible conditions, let $L$ be a surgery presentation of $M$ with color induced from $\omega$. Then \cite{MR3286896} showed that the quantum invariant of $L\cup T$ after normalization is a topological invariant of $(M, T, \omega)$.

In this paper, we follow the method in \cite{MR3286896} to construct a $3$-manifold invariant. The quantum invariant we use is Viro's $\mathfrak{gl}(1\vert 1)$-Alexander polynomial defined in \cite{MR2255851}. Consider a $1$-palette defined by $(B, G)$ where $B$ is a field of characteristic $0$ and $G\subset B$ is an abelian group. There is a category $\mathcal{M}_B$ of finite dimensional modules over a subalgebra $U^1$ of  $U_{q}(\mathfrak{gl}(1 \vert 1))$, the quantum group of the Lie superalgebra $\mathfrak{gl}(1 \vert 1)$. The category $\mathcal{M}_B$ is not semisimple and the objects of which have zero quantum dimensions. Viro defined a functor from the category of trivalent graphs to $\mathcal{M}_B$. For a colored graph $\Gamma$, the Alexander polynomial $\Delta (\Gamma)$ is defined using this functor. 

Now consider a triple $(M, \Gamma, \omega)$, where $M$ is a $3$-manifold, $\Gamma$ is a trivalent graph colored by objects of $\mathcal{M}_B$ and $\omega: H_1(M\backslash \Gamma, \mathbb{Z})\to G$ is a cohomology class. We assume that $(M, \Gamma, \omega)$ satisfies certain compatible conditions. Here is our main result. The definitions of Kirby color and computable surgery presentation will be given in Section 3.3 and Section 4.

\begin{theo} 
\label{mainresult1}
For the $1$-palette $(B, G)$ where $G$ contains $\mathbb{Z}$ but no $\mathbb{Z}/2\mathbb{Z}$ as a subgroup, let $(M, \Gamma, \omega)$ be a compatible triple. Let $L$ be a computable surgery presentation of $(M, \Gamma, \omega)$. Then  
$$\Delta (M, \Gamma, \omega):=\frac{\Delta (L\cup \Gamma)}{2^{r(L)}(-1)^{\sigma_{+}(L)}}$$
is a topological invariant of $(M, \Gamma, \omega)$, where $r(L)$ is the component number of $L$ and $\sigma_{+}(L)$ is the number of positive eigenvalues of the linking matrix of $L$. Here each component $K$ of $L$ has Kirby color $\Omega(\omega([m_K]), 1)$, where $m_K$ is the meridian of $K$. 
\end{theo}

Our strategy is as follows. Instead of proving that $\mathcal{M}_B$ has a relative $G$-modular category structure, we show directly that the value $\Delta (M, \Gamma, \omega)$ is invariant under Kirby moves. So the flavor of this paper is quite combinatorial without involving many algebras. However we believe the existence of a relative $G$-modular category structure on $\mathcal{M}_B$ so that the corresponding invariant is the one given in Theorem \ref{mainresult1}. We hope to discuss this topic in our future work. In the definitions of compatible triple, Kirby color and the proof of Theorem \ref{mainresult1}, we imitate many ideas from \cite{MR3286896}. 

The authors of \cite{MR3286896} discussed in detail how to define the 3-manifold invariant in the context of quantum $\mathfrak{sl}(2)$. For any finite-dimensional simple complex Lie algebra $\mathfrak{g}$, they also showed the existence of relative $G$-modular category associated with certain version of quantum $\mathfrak{g}$. The representation theory for Lie superalgebras is much more complicated than that of Lie algebras. Based on the concept relative $G$-modular category, NP Ha \cite{MR3785790} constructed 3-manifold invariant from quantum group associated with Lie superalgebra $\mathfrak{sl}(2\vert 1)$. It is not clear to us yet whether $\Delta (M, \Gamma, \omega)$ coincides with any known invariant or not.

The structure of the paper is as follows. In Section 2 we review the definition of Viro's $\mathfrak{gl}(1\vert 1)$-Alexander polynomial for trivalent graphs. We calculate examples and recall junction relations, both of which will be used in subsequent sections. In Section 3, we give the definitions of compatible triple and Kirby color, and discuss how the $\mathfrak{gl}(1\vert 1)$-Alexander polynomial changes under Kirby moves. In Section 4, we state the main result and give the proof. In Section 5 we discuss examples and calculations of this invariant. In particular, we take lens spaces $L(7, 1)$ and $L(7, 2)$ as example to show that the invariant $\Delta (M, \Gamma, \omega)$ can distinguish homotopy equivalent manifolds. 

\medskip

\noindent{\bf Acknowledgements}
The authors would like to thank Prof.~Jun Murakami and Dr. Atsuhiko Mizusawa for their helpful discussions. They also would like to thank Prof.~Tetsuya Ito for suggesting us compute our invariant for $L(7, 1)$ and $L(7, 2)$. The first author was partially supported by JSPS KAKENHI Grant Number JP20K14304. The second author was partially supported by JSPS KAKENHI Grant Number 20K03604 and Toyohashi Tech Project of Collaboration with KOSEN.


\medskip

\section{Viro's $\mathfrak{gl}(1\vert 1)$-Alexander polynomial for trivalent graphs}
Viro \cite{MR2255851} defined a functor from the category of colored framed oriented trivalent graphs to the category of finite dimensional modules over a subalgebra $U^1$ of the $q$-deformed universal enveloping superalgebra $U_{q}(\mathfrak{gl}(1 \vert 1))$. Using this functor, in \cite[Sect. 6]{MR2255851}, he defined the $\mathfrak{gl}(1\vert 1)$-Alexander polynomial for a trivalent graph. We recall how this polynomial is calculated. For the algebraic structures of $U^1$ and $U_{q}(\mathfrak{gl}(1 \vert 1))$, please read \cite[$\S$11: Appendix]{MR2255851}. 

\subsection{Colored framed graphs}

A {\it $1$-palette} (see \cite[2.8]{MR2255851}) is a quadruple $$(B, G, W, G\times W \to G),$$ where $B$ is a commutative ring with unit, $G$ is a subgroup of the multiplicative group of $B$, $W$ is a subgroup of the additive group of $B$ which contains the unit $1$ of $B$, and $G\times W \to G: (t, N)\mapsto t^{N}$ is a bilinear map satisfying $t^1=t$ for each $t\in G$. 
In this paper, we consider the case that $B$ is a field of characteristic $0$. Let $G$ is a subgroup of the multiplicative group of $B$, which is abelian, and let $W=\mathbb{Z}$ and $G\times \mathbb{Z} \to G: (t, N)\mapsto t^{N}$. Obviously $(B, G, W, G\times W \to G)$ becomes a $1$-palette. Since $W$ and $G\times \mathbb{Z} \to G$ have specific definitions, we suppress them and use $(B, G)$ to denote the $1$-palette. In this paper, when we say a $1$-palette, we mean a $1$-palette defined in this way.

Let $T$ be an oriented trivalent graph, and let $E$ be the set of edges of $T$. Consider a map which we call a coloring
\begin{eqnarray*}
c=(\mathrm{mul}, \mathrm{wt}): E &\to& G\backslash \{g\in G \mid g^4=1\}\times \mathbb{Z}\\
e &\mapsto& (t, N).
\end{eqnarray*} 
The first number $t=\mathrm{mul}(e)$ is called the {\it multiplicity} and the second number $N=\mathrm{wt}(e)$ is called the {\it weight}. 

Around a vertex, suppose the three edges adjacent to it are colored by $(t_1, N_1)$, $(t_2, N_2)$ and $(t_3, N_3)$. Let $\epsilon_i=-1$ if the $i$-th edge points toward the vertex and $\epsilon_i=1$ otherwise. The coloring $c$ needs to satisfy the following conditions, which are called {\it admissibility conditions} in \cite{MR2255851}:
\begin{align}
\prod_{i=1}^{3} t_{i}^{\epsilon_i}&=1,\hspace{7cm} \label{multi}\\
\sum_{i=1}^{3} \epsilon_i N_i&=-\prod_{i=1}^{3} \epsilon_i.
\end{align}
A vertex is called {\it source} (resp. {\it sink}) if all the adjacent edges have $\epsilon=1$ (resp. $\epsilon=-1$).

Now consider a proper embedding of $T$ into a 3-manifold $M$. We still use $T$ to represent the embedded graph. A {\it framing} of $T$ is an orientable compact surface $F$ embedded in $M$ in which $T$ is sitting as a deformation retract. More precisely, in $F$ each vertex of $T$ is replaced by a disk where the vertex is the center, and each edge of $T$ is replaced by a strip $[0, 1]\times [0, 1]$ where $[0, 1]\times \{0, 1\}$ is attached to the boundaries of its adjacent vertex disks and $\{\frac{1}{2}\}\times [0, 1]$ is the given edge of $T$. 

A {\it framed graph} is a graph with a framing. By an {\it isotopy} of a framed graph we mean an isotopy of the graph in $M$ which extends to an isotopy of the framing. 

For a framed graph, at each source or sink, we can assign an orientation to the boundary of the associated disk, which is regarded as part of the coloring of $T$. Now we are ready to give the following definition.

\begin{defn}
\rm
A {\it colored framed oriented trivalent graph} $\Gamma$ in a 3-manifold $M$ is an oriented trivalent graph $T$ embedded in $M$ with the following three structures:
\begin{itemize}
\item a framing;
\item a coloring on the set of edges which satisfies the admissibility conditions;
\item an orientation of the boundary of the associated disk on each source or sink vertex.
\end{itemize}
\end{defn}

In the following sections, a framed graph means a framed oriented trivalent graph, while a colored framed graph means a colored framed oriented trivalent graph.

When $\Gamma$ is a graph in $S^3$, we can use a graph diagram to represent $\Gamma$, the blackboard framing of which coincides with the framing of $\Gamma$. Around a source or sink vertex, the counter-clockwise orientation is chosen unless otherwise stated.




\subsection{$\mathfrak{gl}(1\vert 1)$-Alexander polynomial}

Let $(B, G)$ be a $1$-palette. 
Suppose $\Gamma$ is a colored framed graph embedded in $S^3$ whose coloring is given by the map $c$ as given in Sect. 2.1. We review the definition of the $\mathfrak{gl}(1\vert 1)$-Alexander polynomial of $\Gamma$, which is denoted by $\Delta(\Gamma)$ or $\Delta(\Gamma; c)$. 

It is known that the pair $(t, N)\in G\backslash \{g\in G \mid g^4=1\}\times \mathbb{Z}$ corresponds to two irreducible $U^1$-modules of dimension $(1 \vert 1)$, which are denoted by $U(t, N)_{+}$ and $U(t, N)_{-}$. These two modules are dual to each other. The module $U(t, N)_{+}$ (resp. $U(t, N)_{-}$) is generated by two elements $e_0$ (boson) and $e_1$ (fermion). For details of their algebraic structures please see Appendix 1 of \cite{MR2255851}. 

Choose a graph diagram of $\Gamma$ in $\mathbb{R}^2$. The diagram divides $\mathbb{R}^2$ into several regions, one of which is unbounded. Choose an edge of $\Gamma$ on the boundary of the unbounded region and cut the edge at a generic point. Suppose the color of the edge is $(t, N)$. Deform the graph diagram under isotopies of $\mathbb{R}^2$ to make it in a Morse position under a given orthogonal coordinate system of $\mathbb{R}^{2}$ so that the two endpoints created by cutting have heights zero and one and the critical points, the crossings, and the vertices of the diagram have different heights between zero and one. Namely after deformation the diagram can be divided into several slices by horizontal lines so that each slice is a disjoint union of trivial vertical segments and one of the six elements in Fig.~\ref{fig2}. Each slice connects a sequence of endpoints on its bottom to a sequence of endpoints on its top. In Example~\ref{example}, we show how the Hopf link is divided into such slices.

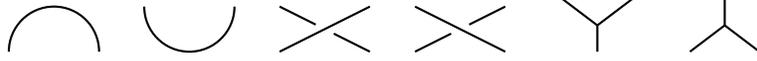
\begin{figure}
\begin{tikzpicture}[baseline=-0.65ex, thick, scale=0.6]
\draw (-1, 0) arc (0:180:1);
\draw (0, 1) arc (180:360:1);
\draw (3, 0) -- (5, 1);
\draw (5, 0) -- (4.2, 0.4);
\draw (3, 1) -- (3.8, 0.6);
\draw (8, 0) -- (6, 1);
\draw (8, 1) -- (7.2, 0.6);
\draw (6, 0) -- (6.8, 0.4);
\end{tikzpicture}\hspace{6mm}
\begin{tikzpicture}[baseline=-0.65ex, thick, scale=0.7]
\draw  (0, 0)  to (0,0.5);
\draw   (0,0.5)   to  (0.66,1);
\draw   (0,0.5)  to  (-0.66,1);
\end{tikzpicture}\hspace{6mm}
\begin{tikzpicture}[baseline=-0.65ex, thick, scale=0.7]
\draw  (0, 0.5)  to (0,1);
\draw   (0,0.5)   to  (0.66,0);
\draw   (0,0.5)   to  (-0.66,0);
\end{tikzpicture}
\caption{Critical points, crossings, and vertices. }
\label{fig2}
\end{figure}

Under Viro's functor, each sequence of endpoints corresponds to the tensor product of irreducible $U^1$-modules of dimension $(1 \vert 1)$, as described below. Suppose the sequence of endpoints is $(p_1, \cdots, p_k)$ for $k\geq 1$, where the subindices represent the $x$-coordinates of the endpoints. Then $(p_1, \cdots, p_k)$ corresponds to the tensor product $$U(t_1, N_1)_{\epsilon_1}\otimes \cdots \otimes U(t_k, N_k)_{\epsilon_k},$$ where $(t_i, N_i)$ is the color of the edge containing $p_i$ and $\epsilon_i=+$ when the edge points upward and $\epsilon_i=-$ otherwise for $1\leq i\leq k$. See Fig.~\ref{fig3}.

\begin{figure}
\centering
\begin{tikzpicture}[baseline=-0.65ex, thick, scale=0.6]
\draw (0, -1)  [->] to (0, 0.5);
\draw (3, -1)  [<-] to (3, 0.5);
\draw (0,-2) node {$U(t, N)_{+}$};
\draw (3,-2) node {$U(t, N)_{-}$};
\draw [dotted] (10, -1)  -- (10, 0.5);
\draw (15, -1)  -- (15, 0.5);
\draw (10,-2) node {$e_0$ (boson)};
\draw (15,-2) node {$e_1$ (fermion)};
\end{tikzpicture}
\caption{Under the coloring $c$, each edge corresponds to an irreducible $U^1$-module. In a state, if an edge is assigned with $e_0$ (resp. $e_1$), we represent it by a dotted (resp. solid) arc.}
\label{fig3}
\end{figure}
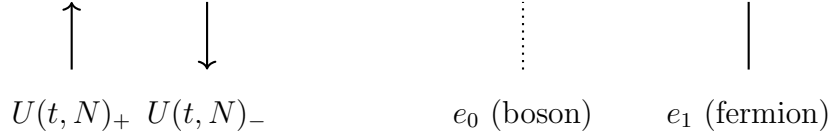

Each slice connects two sequences of endpoints. Under Viro's functor, each slice, read from the bottom to the top, is mapped to a morphism between the corresponding tensor products of irreducible $U^1$-modules.

The morphism is defined in \cite{MR2255851} in the language of Boltzmann weights. Simply speaking, each module $U(t, N)_{+}$ or $U(t, N)_{-}$ has two generators $e_0$ (boson) and $e_1$ (fermion), and therefore 
$U(t_1, N_1)_{\epsilon_1}\otimes \cdots \otimes U(t_k, N_k)_{\epsilon_k}$ is generated by $e_{\delta_1}\otimes e_{\delta_2}\otimes \cdots \otimes e_{\delta_k}$ for $\delta_i=0$ or $1$. The morphism is represented by a matrix under the above choice of generators, and the Boltzmann weights are the entries of the matrix. In Table~\ref{viro1}, we list some Boltzmann weights that we need. For the full table, see Tables 3 and 4 of \cite{MR2255851}.

The composition of two slices (attaching them by identifying the top of the first slice with the bottom of the second slice) corresponds to the composition of their morphisms for $U^1$-modules. As a consequence, the graph diagram in a Morse position with two endpoints of heights zero and one is mapped to a morphism from $U(t, N)_+$ to $U(t, N)_+$ (or $U(t, N)_-$ to $U(t, N)_-$ depending the orientation of $\Gamma$ at the endpoints), which is a scalar of identity (\cite[6.2.A]{MR2255851}). Recall that $(t, N)$ is the color of the edge which was cut. Then multiplying the scalar by the inverse of $t^{2}-t^{-2}$ we get $\Delta(\Gamma)$. In the following paragraphs, we use $(\Gamma)$ to represent the Alexander polynomial of $\Gamma$ when $\Gamma$ is a colored framed graph diagram. 

\begin{ex}
\label{example}
For $u, v\in G\backslash \{g\in G\vert g^4=1\}$, we have
\begin{align*}
 \left ( \begin{tikzpicture}[baseline=-0.65ex, thick, scale=1]
\draw (-0.3,1) [<-]to (0.3,0.4);
\draw (-0.3,0.4) to (-0.1,0.6);
\draw (0.1,0.8) [->] to (0.3,1);
\draw (-0.3,0.4) to (0.3,-0.2);
\draw (-0.3,-0.2) to (-0.1,0);
\draw (0.1,0.2) to (0.3,0.4);
\draw  (-0.3,1) arc (90:270:0.6);
\draw  (0.3,1) arc (90:-90:0.6);
\draw (0.5,-0.5) node {$\scriptstyle (v, V)$};
\draw (-0.5,-0.5) node {$\scriptstyle (u, U)$};
\end{tikzpicture} \right )
&=\frac{1}{v^2-v^{-2}}\left< \begin{tikzpicture}[baseline=-0.65ex, thick, scale=1]
\draw (-0.3,1) to (0.3,0.4);
\draw (-0.3,0.4) to (-0.1,0.6);
\draw (0.1,0.8) to (0.3,1);
\draw (-0.3,0.4) to (0.3,-0.2);
\draw (-0.3,-0.2) to (-0.1,0);
\draw (0.1,0.2) to (0.3,0.4);
\draw  (-0.3,1) arc (0:180:0.3);
\draw  (-0.9,1) [->-]to (-0.9, -0.2);
\draw  (-0.3,-0.2) arc (0:-180:0.3);
\draw (0.3,-0.2) to (0.3,-1);
\draw (0.3,1) [->]to (0.3,1.5);
\draw (0.8,-0.7) node {$\scriptstyle (v, V)$};
\draw (-1.4,-0.5) node {$\scriptstyle (u, U)$};
\draw [thin] (-1.4, 1.5) to (0.9, 1.5);
\draw [thin] (-1.4, 1) to (0.9, 1);
\draw [thin] (-1.4, -1) to (0.9, -1);
\draw [thin] (-1.4, -0.2) to (0.9, -0.2);
\draw [thin] (-1.4, 0.4) to (0.9, 0.4);
\end{tikzpicture}\right >=\frac{-u^{2V}v^{2U}(v^{2}-v^{-2})}{v^2-v^{-2}}=-u^{2V}v^{2U}.\\\\
 \left(\begin{tikzpicture}[baseline=-0.65ex, thick, scale=1]
\draw (-0.3,1) [<-]to (-0.1,0.8);
\draw (0.3,0.4) to (0.1,0.6);
\draw (-0.3,0.4) [->] to (0.3,1);
\draw (-0.3,0.4) to (-0.1,0.2);
\draw (0.3,-0.2) to (0.1,0);
\draw (-0.3,-0.2) to (0.3,0.4);
\draw  (-0.3,1) arc (90:270:0.6);
\draw  (0.3,1) arc (90:-90:0.6);
\draw (0.5,-0.5) node {$\scriptstyle (v, V)$};
\draw (-0.5,-0.5) node {$\scriptstyle (u, U)$};
\end{tikzpicture}\right)&=\frac{1}{v^2-v^{-2}}\left<\begin{tikzpicture}[baseline=-0.65ex, thick, scale=1]
\draw (-0.3,1) to (-0.1,0.8);
\draw (0.3,0.4) to (0.1,0.6);
\draw (-0.3,0.4)  to (0.3,1);
\draw (-0.3,0.4) to (-0.1,0.2);
\draw (0.3,-0.2) to (0.1,0);
\draw (-0.3,-0.2) to (0.3,0.4);
\draw  (-0.3,1) arc (0:180:0.3);
\draw  (-0.9,1) [->-]to (-0.9, -0.2);
\draw  (-0.3,-0.2) arc (0:-180:0.3);
\draw (0.3,-0.2) to (0.3,-1);
\draw (0.3,1) [->]to (0.3,1.5);
\draw (0.8,-0.7) node {$\scriptstyle (v, V)$};
\draw (-1.4,-0.5) node {$\scriptstyle (u, U)$};
\draw [thin] (-1.4, 1.5) to (0.9, 1.5);
\draw [thin] (-1.4, 1) to (0.9, 1);
\draw [thin] (-1.4, -1) to (0.9, -1);
\draw [thin] (-1.4, -0.2) to (0.9, -0.2);
\draw [thin] (-1.4, 0.4) to (0.9, 0.4);
\end{tikzpicture}\right>=\frac{u^{-2V}v^{-2U}(v^{2}-v^{-2})}{v^2-v^{-2}}=u^{-2V}v^{-2U}.
\end{align*}
Here $\langle D\rangle$ denotes the scalar defined the by the tangle $D$.
\end{ex}
\begin{proof}
We view the left-hand diagram in the first row as a morphism from $U(v, V)_{+}$ to itself, which is a scalar of the indentity. To determine the scalar, it is enough to calculate the image of the generator $e_0$. The diagram is divided into 4 slices, each of which contains exactly one critical point or crossing. By the Boltzmann weights in Table 1, we have the following calculation.  
\begin{eqnarray*}
e_0&\mapsto& u^{-2}e_0\otimes e_0\otimes e_0+e_1\otimes e_1\otimes e_0\\
&\mapsto& u^{-2}v^{-1+U}u^{-1+V}e_0\otimes e_0\otimes e_0+v^{1+U}u^{-1+V}e_1\otimes e_0\otimes e_1+\frac{1-v^4}{v^{1-U}u^{1-V}}e_1\otimes e_1\otimes e_0\\
&\mapsto& u^{-2}v^{-2+2U}u^{-2+2V}e_0\otimes e_0\otimes e_0+v^{2+2U}u^{-2+2V}e_1\otimes e_1\otimes e_0\\
&&+\frac{(1-v^4)(1-u^4)}{v^{2-2U}u^{2-2V}} e_1\otimes e_1\otimes e_0+\frac{1-v^4}{v^{1-U}u^{1-V}}v^{-1+U}u^{1+V}e_1\otimes e_0\otimes e_1\\
&\mapsto& [u^{-2}v^{-2+2U}u^{-2+2V}-u^{-2}v^{2+2U}u^{-2+2V}-u^{-2}\frac{(1-v^4)(1-u^4)}{v^{2-2U}u^{2-2V}}]e_0\\
&=& u^{2V}v^{2U-2}(1-v^{4})e_0=-u^{2V}v^{2U}(v^2-v^{-2})e_0.
\end{eqnarray*}
After dividing $v^2-v^{-2}$, we get the desired result. The calculation of the other case can be conducted in the same way.
\end{proof}



\begin{table}
\begin{center}
\begin{tabular} {|c|c|c|c|c|} \hline
&\begin{tikzpicture}[baseline=-0.65ex, thick]
\draw [dotted] (0, 0) [<-] to [out=270,in=180] (0.5, -0.5) to [out=0,in=270] (1,0);
\draw (1.7, -0.3) node {$(u, U)$};
\draw (0.7, -1) node {$1$};
\end{tikzpicture}&
\begin{tikzpicture}[baseline=-0.65ex, thick]
\draw  (0, 0) [<-] to [out=270,in=180] (0.5, -0.5) to [out=0,in=270] (1,0);
\draw (1.7, -0.3) node {$(u, U)$};
\draw (0.7, -1) node {$-u^{2}$};
\end{tikzpicture}&
\begin{tikzpicture}[baseline=-0.65ex, thick]
\draw [dotted] (0, 0) to [out=270,in=180] (0.5, -0.5)  [->] to [out=0,in=270] (1,0);
\draw (1.7, -0.3) node {$(u, U)$};
\draw (0.7, -1) node {$u^{-2}$};
\end{tikzpicture}
&
\begin{tikzpicture}[baseline=-0.65ex, thick]
\draw (0, 0) to [out=270,in=180] (0.5, -0.5) [->] to [out=0,in=270]  (1,0);
\draw (1.7, -0.3) node {$(u, U)$};
\draw (0.7, -1) node {$1$};
\end{tikzpicture}\\
\hline
&\begin{tikzpicture}[baseline=-0.65ex, thick]
\draw [dotted] (0, 0) [<-] to [out=90,in=180] (0.5, 0.5) to [out=0,in=90] (1,0);
\draw (1.7, 0.3) node {$(u, U)$};
\draw (0.7, -0.5) node {$1$};
\end{tikzpicture}&
\begin{tikzpicture}[baseline=-0.65ex, thick]
\draw  (0, 0) [<-] to [out=90,in=180] (0.5, 0.5) to [out=0,in=90] (1,0);
\draw (1.7, 0.3) node {$(u, U)$};
\draw (0.7, -0.5) node {$-u^{-2}$};
\end{tikzpicture}&
\begin{tikzpicture}[baseline=-0.65ex, thick]
\draw [dotted] (0, 0)  to [out=90,in=180] (0.5, 0.5) [->] to [out=0,in=90] (1,0);
\draw (1.7, 0.3) node {$(u, U)$};
\draw (0.7, -0.5) node {$u^{2}$};
\end{tikzpicture}
&
\begin{tikzpicture}[baseline=-0.65ex, thick]
\draw  (0, 0)  to [out=90,in=180] (0.5, 0.5) [->] to [out=0,in=90] (1,0);
\draw (1.7, 0.3) node {$(u, U)$};
\draw (0.7, -0.5) node {$1$};
\end{tikzpicture}\\
\hline
\hline
\vspace{-2mm}&&&&\\
     & \begin{tikzpicture}[baseline=-0.65ex, thick]
\draw [dotted] (0, 0)  to (1,1);
\draw  [dotted] (1,0)  to  (0,1);
\end{tikzpicture}
     & \begin{tikzpicture}[baseline=-0.65ex, thick]
\draw [dotted] (0, 0)  to (1,1);
\draw  (1,0)  to  (0,1);
\end{tikzpicture}
      & \begin{tikzpicture}[baseline=-0.65ex, thick]
\draw  (0, 0)  to (1,1);
\draw  [dotted] (1,0)  to  (0,1);
\end{tikzpicture}
& \begin{tikzpicture}[baseline=-0.65ex, thick]
\draw  (0, 0)  to (1,1);
\draw  (1,0)  to  (0,1);
\end{tikzpicture}
       \\ \hline 
 \vspace{-2mm}     &&&&\\
\begin{tikzpicture}[baseline=-0.65ex, thick]
\draw  (0, 0) [->]  to (1,1);
\draw  (1,0)  to  (0.6,0.4);
\draw  (0.4,0.6) [->]  to  (0,1);
\draw (1.6, 0) node {$(u, U)$};
\draw (-0.6, 0) node {$(v, V)$};
\end{tikzpicture}
 & $v^{1-U}u^{1-V}$   & $v^{-1-U}u^{1-V}$ & $v^{1-U}u^{-1-V}$ &$-v^{-1-U}u^{-1-V}$ \\ 
\hline 
 \vspace{-2mm}     &&&&\\
\begin{tikzpicture}[baseline=-0.65ex, thick]
\draw  (0.6, 0.6) [->]  to (1,1);
\draw  (1,0) [->] to  (0,1);
\draw  (0,0)  to  (0.4,0.4);
\draw (1.6, 0) node {$(u, U)$};
\draw (-0.6, 0) node {$(v, V)$};
\end{tikzpicture}
 & $v^{U-1}u^{V-1}$   & $v^{U+1}u^{V-1}$ & $v^{U-1}u^{V+1}$ &$-v^{U+1}u^{V+1}$ \\ 
\hline 
\hline
\vspace{-2mm}&&&&\\
     & \begin{tikzpicture}[baseline=-0.65ex, thick]
\draw (0, 0)  to (0.5,0.5);
\draw [dotted] (0.5, 0.5)  to (1,1);
\draw (0.5, 0.5) to (0, 1);
\draw  [dotted] (1,0)  to  (0.5,0.5);
\end{tikzpicture}
     & \begin{tikzpicture}[baseline=-0.65ex, thick]
\draw [dotted] (0, 0)  to (0.5,0.5);
\draw  (0.5, 0.5)  to (1,1);
\draw [dotted] (0.5, 0.5) to (0, 1);
\draw  (1,0)  to  (0.5,0.5);
\end{tikzpicture}
      & \begin{tikzpicture}[baseline=-0.65ex, thick]
\draw (0, 0)  to (0.5,0.5);
\draw [dotted]  (0.5, 0.5)  to (1,1);
\draw [dotted] (0.5, 0.5) to (0, 1);
\draw  (1,0)  to  (0.5,0.5);
\end{tikzpicture}
& \begin{tikzpicture}[baseline=-0.65ex, thick]
\draw  [dotted] (0, 0)  to (0.5,0.5);
\draw  (0.5, 0.5)  to (1,1);
\draw   (0.5, 0.5) to (0, 1);
\draw  [dotted] (1,0)  to  (0.5,0.5);
\end{tikzpicture}
       \\ \hline 
 \vspace{-2mm}     &&&&\\
\begin{tikzpicture}[baseline=-0.65ex, thick]
\draw  (0, 0) [->]  to (1,1);
\draw  (1,0)  to  (0.6,0.4);
\draw  (0.4,0.6) [->]  to  (0,1);
\draw (1.6, 0) node {$(u, U)$};
\draw (-0.6, 0) node {$(v, V)$};
\end{tikzpicture}
 & $0$   & $\displaystyle \frac{v^{4}-1}{v^{1+U}u^{1+V}}$ & $0$ &$0$ \\ 
 \vspace{-2mm}     &&&&\\
\hline 
\vspace{-2mm}     &&&&\\
\begin{tikzpicture}[baseline=-0.65ex, thick]
\draw  (0.6, 0.6) [->]  to (1,1);
\draw  (1,0) [->] to  (0,1);
\draw  (0,0)  to  (0.4,0.4);
\draw (1.6, 0) node {$(u, U)$};
\draw (-0.6, 0) node {$(v, V)$};
\end{tikzpicture}
 & $\displaystyle \frac{1-u^{4}}{v^{1-U}u^{1-V}}$   & $0$ & $0$ &$0$ \\ 
     &&&&\\
\hline
\end{tabular}
\vspace{3mm}
\caption{Boltzmann weights for critical points, half-twist symbols and two types of crossings from Viro's Table 3.}
\label{viro1}
\end{center}
\end{table}

Viro's functor satisfies the following relations, which will be used in Section 3. See \cite[5.2B]{MR2255851}.

\begin{lemma}[Junction relations]
\label{junction}
For $t\in G$ satisfying $t^4\neq 1$, let $\displaystyle d(t)=\frac{1}{t^2-t^{-2}}$. When $u^4v^4\neq 1$ we have
\begin{align*}
\left(\begin{tikzpicture}[baseline=-0.65ex, thick, scale=1]
\draw (0,-1) [<-] to (0,1);
\draw (1,-1) [<-] to (1,1);
\draw (-0,-1.25) node {$\scriptstyle(u, U)$};
\draw (1, -1.25) node {$\scriptstyle(v, V)$};
\end{tikzpicture}\right)=d(uv)
\left(\begin{tikzpicture}[baseline=-0.65ex, thick, scale=1]
\draw (0,-1) [<-] to  (0.5,-0.5);
\draw (0.5, -0.5)  [->]  to  (1,-1);
\draw (0.5, 0.5)    [-<-] to (0,1);
\draw (0.5,0.5) [-<-] to  (1,1);
\draw (0.5,-0.5) [-<-] to  (0.5,0.5);
\draw (0, -1.25) node {$\scriptstyle(u, U)$};
\draw (1, -1.25) node {$\scriptstyle(v, V)$};
\draw (1.5, 0) node {${ \scriptstyle (uv, U+V-1)}$};
\draw (0, 1.25) node {$\scriptstyle(u, U)$};
\draw (1, 1.25) node {$\scriptstyle(v, V)$};
\end{tikzpicture}\right)
-d(uv)
\left(\begin{tikzpicture}[baseline=-0.65ex, thick, scale=1]
\draw (0,-1) [<-] to  (0.5,-0.5);
\draw (0.5, -0.5)  [->]  to  (1,-1);
\draw (0.5, 0.5)    [-<-] to (0,1);
\draw (0.5,0.5) [-<-] to  (1,1);
\draw (0.5,-0.5) [->-] to  (0.5,0.5);
\draw (0, -1.25) node {$\scriptstyle(u, U)$};
\draw (1, -1.25) node {$\scriptstyle(v, V)$};
\draw (2, 0) node {${\scriptstyle (u^{-1}v^{-1}, -U-V-1)}$};
\draw (0, 1.25) node {$\scriptstyle(u, U)$};
\draw (1, 1.25) node {$\scriptstyle(v, V)$};
\end{tikzpicture}\right).
\end{align*}
When $u^4v^{-4}\neq 1$ we have
\begin{align*}
\left(\begin{tikzpicture}[baseline=-0.65ex, thick, scale=1]
\draw (0,-1) [<-] to (0,1);
\draw (1,-1) [->] to (1,1);
\draw (-0,-1.25) node {$\scriptstyle(u, U)$};
\draw (1, -1.25) node {$\scriptstyle(j, J)$};
\end{tikzpicture}\right)=d(uv^{-1})
\left(\begin{tikzpicture}[baseline=-0.65ex, thick, scale=1]
\draw (0,-1) [<-] to  (0.5,-0.5);
\draw (0.5, -0.5)  [-<-]  to  (1,-1);
\draw (0.5, 0.5)    [-<-] to (0,1);
\draw (0.5,0.5) [->] to  (1,1);
\draw (0.5,-0.5) [-<-] to  (0.5,0.5);
\draw (0, -1.25) node {$\scriptstyle(u, U)$};
\draw (1, -1.25) node {$\scriptstyle(v, V)$};
\draw (1.7, 0) node {${\scriptstyle (uv^{-1}, U-V+1)}$};
\draw (0, 1.25) node {$\scriptstyle(u, U)$};
\draw (1, 1.25) node {$\scriptstyle(v, V)$};
\end{tikzpicture}\right)
-d(uv^{-1})
\left(\begin{tikzpicture}[baseline=-0.65ex, thick, scale=1]
\draw (0,-1) [<-] to  (0.5,-0.5);
\draw (0.5, -0.5)  [-<-]  to  (1,-1);
\draw (0.5, 0.5)    [-<-] to (0,1);
\draw (0.5,0.5) [->] to  (1,1);
\draw (0.5,-0.5) [->-] to  (0.5,0.5);
\draw (0, -1.25) node {$\scriptstyle(u, U)$};
\draw (1, -1.25) node {$\scriptstyle(v, V)$};
\draw (1.8, 0) node {${\scriptstyle (u^{-1}v, -U+V+1)}$};
\draw (0, 1.25) node {$\scriptstyle(u, U)$};
\draw (1, 1.25) node {$\scriptstyle(v, V)$};
\end{tikzpicture}\right).
\end{align*}
\end{lemma}

\medskip

\section{Properties of the Alexander polynomial under Kirby moves}
\subsection{Cohomology classes}
We review a characterization of cohomology classes given in \cite[Sect. 2.3]{MR3286896}. Let $M$ be a $3$-manifold, let $T$ be a framed graph in $M$. Suppose $L$ is an oriented framed link in $S^3$ which is a surgery presentation for $M$. Since $T$ is disjoint from $L$, we also view $T$ as a graph in $S^3$ before the surgery. 

Now we consider diagrams of $L$ and $T$, which are still denoted by $L$ and $T$. Let $e_1, e_2, \cdots, e_r$ be the components of $L$, and $e_{r+1}, e_{r+2}, \cdots, e_{r+s}$ be the oriented edges of $T$. For two different components $e_i$ and $e_j$ in $L$ ($1 \leq i, j\leq r$), let $lk_{ij}=lk(e_i, e_j)$ denote the linking number of $e_i$ and $e_j$. Namely, it is half of the sum of signs of all the crossings between $e_i$ and $e_j$. Let $lk_{ii}=lk(e_i, e_i)$ ($1 \leq i \leq r$) be the framing of $e_i$. Namely it is the sum of signs of self-crossings of $e_i$ (since we use blackboard framing). It is well-known that $lk_{ij}$ does not depend on the diagram we choose. The matrix $(lk_{ij})_{1\leq i, j \leq r}$ is called the {\it linking matrix} of $L$. 

For a component $e_i$ of $L$ and an edge $e_j$ of $T$, we define the linking number $lk_{ij}=lk(e_i, e_j)$ to be the number of all the crossings of type \begin{tikzpicture}[baseline=-0.65ex, thick, scale=0.5]
\draw  (0, 0) [->]  to (1,1);
\draw  (1,0)  to  (0.6,0.4);
\draw  (0.4,0.6) [->]  to  (0,1);
\draw (1.5, 0) node {$e_i$};
\draw (-0.5, 0) node {$e_j$};
\end{tikzpicture} minus the number of crossings of type 
\begin{tikzpicture}[baseline=-0.65ex, thick, scale=0.5]
\draw  (0.6, 0.6) [->]  to (1,1);
\draw  (1,0) [->] to  (0,1);
\draw  (0,0)  to  (0.4,0.4);
\draw (1.5, 0) node {$e_j$};
\draw (-0.5, 0) node {$e_i$};
\end{tikzpicture} between $e_i$ and $e_j$. Note that this number depends on the diagrams of $L$ and $T$.
 
Let $M\backslash T$ be the complement of $T$ in $M$. The first homology group $H_{1}(M\backslash T, \mathbb{Z})$ has a presentation
$$H_1(M\backslash T, \mathbb{Z})=\left <\begin{array}{l|l}
  & \forall 1\leq i \leq r,\sum_{j=1}^{r+s} lk_{ij}[m_j]=0;\\
\{ [m_i]\}_{1\leq i\leq r+s} & \forall v:\text{vertex of $T$}, r_v=0;\\
&\forall 1\leq i, j \leq r+s, [m_i]+[m_j]=[m_j]+[m_i]
\end{array}
\right >,
$$
where $m_i$ is the oriented meridian of $e_i$, and for a vertex $v$ of $T$, $r_v$ is the sum of meridians of the edges entering $v$ minus the sum of meridians of the edges outgoing from $v$. Note that for each $1\leq i \leq r$, $\sum_{j=r+1}^{r+s}lk_{ij}[m_j]$ does not depend on the choice of diagram of $L$ and $T$ and is a well-defined value.

Let $G$ be an abelian group. Then the cohomology class  $$\omega\in H^1(M\backslash T, G)\cong \mathrm{Hom} (H_1(M\backslash T, \mathbb{Z}), G)$$ is uniquely determined by the images of $[m_i]'s$ under $\omega$.

\subsection{Kirby calculus}
We review basic facts about Kirby calculus, which can be found, for instance in \cite[5.1]{MR3286896}. Kirby \cite{MR467753} showed that any $3$-manifold can be obtained by doing surgeries along a framed link in $S^3$. Such a link is called the surgery presentation of the given $3$-manifold. There are two types of moves connecting surgery presentations, which are called blow up/down moves and handle-slide move. See Fig.~\ref{f3} and Fig.~\ref{f4}.

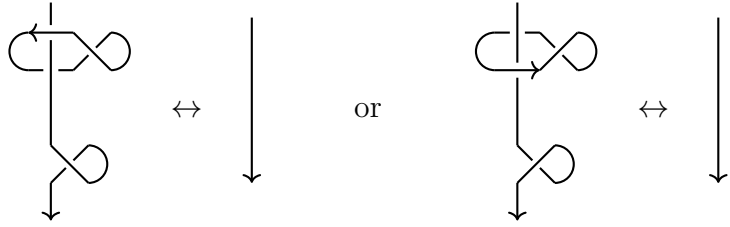
\begin{figure}
\begin{align*}
\begin{tikzpicture}[baseline=-0.65ex, thick, scale=1]
\draw (0,-0.5) to (0,0.9);
\draw (0,1.1) to (0,1.4);
\draw (-0.3,1) [<-]to (0.3,1);
\draw (-0.3,0.5) to (-0.1,0.5);
\draw (0.1,0.5) to (0.3,0.5);
\draw  (-0.3,1) arc (90:270:0.25);
\draw (0.3,1) to (0.8,0.5);
\draw (0.3,0.5) to (0.5,0.7);
\draw (0.8,1) to (0.6,0.8);
\draw (0,-0.5) to (0.5,-1);
\draw (0.8,0.5) arc (-90:90:0.25);
\draw (0,-1) to (0.2,-0.8);
\draw (0.3,-0.7) to (0.5,-0.5);
\draw (0.5,-1) arc (-90:90:0.25);
\draw (0,-1) [->]  to (0,-1.5);
\end{tikzpicture} \quad \leftrightarrow \quad \begin{tikzpicture}[baseline=-0.65ex, thick, scale=1]
\draw (0.5,-1) [<-]  to (0.5,1.2);
\end{tikzpicture}\quad\quad\quad \text{or}\quad\quad\quad
\begin{tikzpicture}[baseline=-0.65ex, thick, scale=1]
\draw (0,-0.5) to (0,0.4);
\draw (0,0.6) to (0,1.4);
\draw (-0.3,0.5) [->]to (0.3,0.5);
\draw (-0.3,1) to (-0.1,1);
\draw (0.1,1) to (0.3,1);
\draw  (-0.3,1) arc (90:270:0.25);
\draw (0.3,1) to (0.5,0.8);
\draw (0.6,0.7) to (0.8,0.5);
\draw (0.3,0.5) to (0.8,1);
\draw (0,-0.5) to (0.2,-0.7);
\draw (0.3,-0.8) to (0.5,-1);
\draw (0.8,0.5) arc (-90:90:0.25);
\draw (0,-1) to (0.5,-0.5);
\draw (0.5,-1) arc (-90:90:0.25);
\draw (0,-1) [->]  to (0,-1.5);
\end{tikzpicture}\quad \leftrightarrow \quad \begin{tikzpicture}[baseline=-0.65ex, thick, scale=1]
\draw (0.5,-1) [<-]  to (0.5,1.2);
\end{tikzpicture}
\end{align*}
\caption{Blow up/down moves}
\label{f3}
\end{figure}

\begin{figure}
\begin{eqnarray*}
\begin{tikzpicture}[baseline=-0.65ex, thick, scale=0.8]
\draw (0.5,-1.2)  to (0.5,1.4);
\draw (0.5,-1.5) node {$e_i$};
\draw [dotted](2.2,-0.8) arc (-90:90:1);
\draw (2.2,-0.8) arc (-90:-270:1);
\draw (2.2,-1.2) node {$e_j$};
\end{tikzpicture}\quad\quad \leftrightarrow \quad\quad 
\begin{tikzpicture}[baseline=-0.65ex, thick, scale=1]
\draw [dotted] (0.3, -0.8) arc (-90:90:0.55);
\draw  (0.3, -0.8) arc (-90:-270:0.55);
\draw [dotted] (0.3, -1.1) arc (-90:90:0.9);
\draw (0.3, -1.1) to [out=180,in=270] (-0.5, -0.5) to [out=180,in=0] (-1, -0.5) to [out=270,in=90] (-1, -1.5);
\draw (-1, 1) to (-1, 0) to (-0.5, 0) to [out=90,in=180] (0.3, 0.7);
\draw (-1.2, -1.25) node {$e_i$};
\draw (0.2, 0) node {$e_j$};
\end{tikzpicture}
\end{eqnarray*}
\caption{Handle-slide move of $e_i$ along $e_j$.}
\label{f4}
\end{figure}
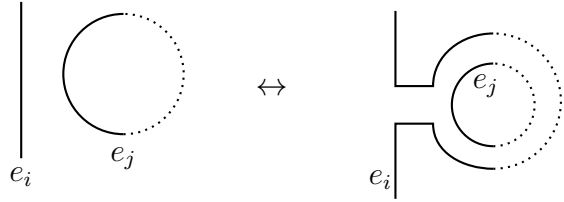

\begin{theo}[Theorem 5.2 in \cite{MR3286896}]
\label{kirby}
Let $M_1$ and $M_2$ be two $3$-manifolds and $T_1\subset M_1$ and $T_2\subset M_2$ be embedded framed graphs. Let $f:M_1\to M_2$ be an orientation preserving diffeomorphism such that $f(T_1)=T_2$. Let $L_i\subset S^3$ be a surgery presentation of $M_i$ which is disjoint from $T_i$ for $i=1, 2$. Then $f$ is isotopic to the diffeomorphisms induced by a finite sequence of moves: 
$$L_1=L^0\xrightarrow{k_1} L^1 \xrightarrow{k_2}\cdots \xrightarrow{k_r}L^{k}=L_2,$$
where each $k_i$ ($1\leq i \leq r$) is one of the following moves.
\begin{enumerate}
\item handle-slide move of a component/edge of $L^{i-1}\cup T_1$ along a component of $L^{i-1}$;
\item blow up/down move along a component/edge of $L^{i-1}\cup T_1$, where the circle component which appears or disappears during this move must be a component of the surgery presentation.
\end{enumerate}
\end{theo}

In general, the composition of a handle-slide move and a blow-up move can be realized by the composition of a blow-up move and several handle-slide moves (possibly one), and the composition of a blow-down move and a handle-slide move can be realized by the composition of several handle-slide moves (possibly one) and a blow-down move. A blow-up move after a blow-down move can be done before the blow-down move without changing the result. Therefore for the sequence of moves connecting $L_1\cup \Gamma_1$ and $L_2 \cup \Gamma_2$, we can assume that all the blow-up moves are at the beginning and all the blow-down moves are at the end.

\subsection{Lemmas}
In \cite[7.5]{MR2255851}, Viro discussed how the Alexander polynomial changes when the weights change. As a special situation, we have the following lemma. For a colored framed graph $\Lambda \subset S^3$ and a knot component $K\subset \Lambda$, we define {\it the colored linking number} of $K$ with $\Lambda$ as 
$$clk(K, \Lambda):=\prod_{\text{$e$: edge of $\Lambda$}}t_e^{lk(K, e)},$$
where $lk(K, e)$ is the linking number as defined in Section 3.1, and $t_e$ is the multiplicity of $e$. Due to the admissibility condition \eqref{multi} for multiplicities, $clk(K, \Lambda)$ is well-defined.

\begin{lemma}
\label{zero}
Let $\Lambda$ be a colored framed graph in $S^3$. Let $K$ be a knot component of $\Lambda$ satisfying the condition that $clk(K, \Lambda)=1$. Then $\Delta (\Lambda)$ does not depend on the choice of the weight on $K$.
\end{lemma}
\begin{proof}
We consider a graph diagram of $\Lambda$, which is still denoted by $\Lambda$. Suppose $c$ is the color of $\Lambda$ and $c(K)=(t, N)$. Let $c'$ be a color which are the same as $c$ except that at the component $K$ we have $c'(K)=(t, N+J)$. A straightforward comparison of Boltzmann weights tells us the following facts. At a positive crossing, if the weight at one of the two strands changes from $N$ to $N+J$, the contribution of the Boltzmann weight makes a change of $s^{-J}$, where $s$ is the multiplicity of the other strand at that crossing. For a negative crossing, the change is $s^{J}$. Therefore
\begin{eqnarray*} 
\Delta (\Lambda; c')&=&\Delta (\Lambda; c)\prod_{\text{$e$: edge of $\Lambda$}} t_e^{-2lk(K, e)J}=\Delta (\Lambda; c)\left (\prod_{\text{$e$: edge of $\Lambda$}} t_e^{2lk(K, e)}\right )^{-J}\\&=&\Delta (\Lambda; c)1^{-J}=\Delta (\Lambda; c),
\end{eqnarray*}
where $t_e$ is the multiplicity of the edge $e$.
\end{proof}

\begin{defn}
\rm
For $(t, N)\in G\backslash \{g\in G \mid g^4=1\}\times \mathbb{Z}$, if one component of a link has {\it Kirby color} $\Omega (t, N)$, the Alexander polynomial is calculated as follows:
\begin{align*}
\left(\begin{tikzpicture}[baseline=-0.65ex, thick, scale=1]
\draw (0.5,-0.5) [<-]  to (0.5,1.2);
\draw (0.5,-0.8) node {$\Omega(t, N)$};
\end{tikzpicture}\right):=d(t)\left(\begin{tikzpicture}[baseline=-0.65ex, thick, scale=1]
\draw (0.5,-0.5) [<-]  to (0.5,1.2);
\draw (0.5,-0.8) node {$(t, N)$};
\end{tikzpicture}\right)-d(t)\left(\begin{tikzpicture}[baseline=-0.65ex, thick, scale=1]
\draw (0.5,-0.5) [->]  to (0.5,1.2);
\draw (0.5,-0.8) node {$(t^{-1}, 2-N)$};
\end{tikzpicture}\right),
\end{align*}
where $d(t)=\frac{1}{t^2-t^{-2}}$.
For a strand with Kirby color $\Omega (t, N)$, its {\it multiplicity} is defined to be $t$.
\end{defn}

It is easy to see that if a knot $K\subset S^3$ has Kirby color $\Omega (t, N)$, we have
$$\Delta (K; \Omega (t, N))=\Delta (-K; \Omega (t^{-1}, 2-N)),$$
where $-K$ is the same knot $K$ with opposite orientation.

Now, we discuss how the Alexander polynomial changes under blow-up/down moves when the circle component has a Kirby color.

\begin{lemma}
\label{up-down}
\begin{align*}
\left(\begin{tikzpicture}[baseline=-0.65ex, thick, scale=1]
\draw (0,-0.5) to (0,0.9);
\draw (0,1.1) to (0,1.4);
\draw (-0.3,1) [<-]to (0.3,1);
\draw (-0.3,0.5) to (-0.1,0.5);
\draw (0.1,0.5) to (0.3,0.5);
\draw  (-0.3,1) arc (90:270:0.25);
\draw (0.3,1) to (0.8,0.5);
\draw (0.3,0.5) to (0.5,0.7);
\draw (0.8,1) to (0.6,0.8);
\draw (0,-0.5) to (0.5,-1);
\draw (0.8,0.5) arc (-90:90:0.25);
\draw (0,-1) to (0.2,-0.8);
\draw (0.3,-0.7) to (0.5,-0.5);
\draw (0.5,-1) arc (-90:90:0.25);
\draw (0,-1) [->]  to (0,-1.5);
\draw (0.6,-1.25) node {$(t, N)$};
\draw (-0.8,0.2) node {$\Omega(t, J)$};
\end{tikzpicture}\right)=2\left(\begin{tikzpicture}[baseline=-0.65ex, thick, scale=1]
\draw (0.5,-1) [<-]  to (0.5,1.2);
\draw (0.5,-1.25) node {$(t, N)$};
\end{tikzpicture}\right),\quad
\left(\begin{tikzpicture}[baseline=-0.65ex, thick, scale=1]
\draw (0,-0.5) to (0,0.4);
\draw (0,0.6) to (0,1.4);
\draw (-0.3,0.5) [->]to (0.3,0.5);
\draw (-0.3,1) to (-0.1,1);
\draw (0.1,1) to (0.3,1);
\draw  (-0.3,1) arc (90:270:0.25);
\draw (0.3,1) to (0.5,0.8);
\draw (0.6,0.7) to (0.8,0.5);
\draw (0.3,0.5) to (0.8,1);
\draw (0,-0.5) to (0.2,-0.7);
\draw (0.3,-0.8) to (0.5,-1);
\draw (0.8,0.5) arc (-90:90:0.25);
\draw (0,-1) to (0.5,-0.5);
\draw (0.5,-1) arc (-90:90:0.25);
\draw (0,-1) [->]  to (0,-1.5);
\draw (0.6,-1.25) node {$(t, N)$};
\draw (-0.8,0.2) node {$\Omega(t, J)$};
\end{tikzpicture}\right)=-2\left(\begin{tikzpicture}[baseline=-0.65ex, thick, scale=1]
\draw (0.5,-1) [<-]  to (0.5,1.2);
\draw (0.6,-1.25) node {$(t, N)$};
\end{tikzpicture}\right).
\end{align*}
\end{lemma}

\begin{proof} The equalities follow from the definition of Kirby color and the following facts.
\begin{align*}
\left(\begin{tikzpicture}[baseline=-0.65ex, thick, scale=1]
\draw (0,-0.5) to (0,0.9);
\draw (0,1.1) to (0,1.4);
\draw (-0.3,1) [<-]to (0.3,1);
\draw (-0.3,0.5) to (-0.1,0.5);
\draw (0.1,0.5) to (0.3,0.5);
\draw  (-0.3,1) arc (90:270:0.25);
\draw (0.3,1) to (0.8,0.5);
\draw (0.3,0.5) to (0.5,0.7);
\draw (0.8,1) to (0.6,0.8);
\draw (0,-0.5) to (0.5,-1);
\draw (0.8,0.5) arc (-90:90:0.25);
\draw (0,-1) to (0.2,-0.8);
\draw (0.3,-0.7) to (0.5,-0.5);
\draw (0.5,-1) arc (-90:90:0.25);
\draw (0,-1) [->]  to (0,-1.5);
\draw (0.6,-1.25) node {$(t, N)$};
\draw (-0.5,0.2) node {$t$};
\end{tikzpicture}\right)=-\left(\begin{tikzpicture}[baseline=-0.65ex, thick, scale=1]
\draw (0,-0.5) to (0,0.9);
\draw (0,1.1) to (0,1.4);
\draw (-0.3,1) [->]to (0.3,1);
\draw (-0.3,0.5) to (-0.1,0.5);
\draw (0.1,0.5) to (0.3,0.5);
\draw  (-0.3,1) arc (90:270:0.25);
\draw (0.3,1) to (0.8,0.5);
\draw (0.3,0.5) to (0.5,0.7);
\draw (0.8,1) to (0.6,0.8);
\draw (0,-0.5) to (0.5,-1);
\draw (0.8,0.5) arc (-90:90:0.25);
\draw (0,-1) to (0.2,-0.8);
\draw (0.3,-0.7) to (0.5,-0.5);
\draw (0.5,-1) arc (-90:90:0.25);
\draw (0,-1) [->]  to (0,-1.5);
\draw (0.6,-1.25) node {$(t, N)$};
\draw (-0.5,0.2) node {$t^{-1}$};
\end{tikzpicture}\right)=(t^{2}-t^{-2})
\left(\begin{tikzpicture}[baseline=-0.65ex, thick, scale=1]
\draw (0.5,-1) [<-]  to (0.5,1.2);
\draw (0.6,-1.25) node {$(t, N)$};
\end{tikzpicture}\right),\\
-\left(\begin{tikzpicture}[baseline=-0.65ex, thick, scale=1]
\draw (0,-0.5) to (0,0.4);
\draw (0,0.6) to (0,1.4);
\draw (-0.3,0.5) [->]to (0.3,0.5);
\draw (-0.3,1) to (-0.1,1);
\draw (0.1,1) to (0.3,1);
\draw  (-0.3,1) arc (90:270:0.25);
\draw (0.3,1) to (0.5,0.8);
\draw (0.6,0.7) to (0.8,0.5);
\draw (0.3,0.5) to (0.8,1);
\draw (0,-0.5) to (0.2,-0.7);
\draw (0.3,-0.8) to (0.5,-1);
\draw (0.8,0.5) arc (-90:90:0.25);
\draw (0,-1) to (0.5,-0.5);
\draw (0.5,-1) arc (-90:90:0.25);
\draw (0,-1) [->]  to (0,-1.5);
\draw (0.6,-1.25) node {$(t, N)$};
\draw (-0.5,0.2) node {$t$};
\end{tikzpicture}\right)=\left(\begin{tikzpicture}[baseline=-0.65ex, thick, scale=1]
\draw (0,-0.5) to (0,0.4);
\draw (0,0.6) to (0,1.4);
\draw (-0.3,0.5) [<-]to (0.3,0.5);
\draw (-0.3,1) to (-0.1,1);
\draw (0.1,1) to (0.3,1);
\draw  (-0.3,1) arc (90:270:0.25);
\draw (0.3,1) to (0.5,0.8);
\draw (0.6,0.7) to (0.8,0.5);
\draw (0.3,0.5) to (0.8,1);
\draw (0,-0.5) to (0.2,-0.7);
\draw (0.3,-0.8) to (0.5,-1);
\draw (0.8,0.5) arc (-90:90:0.25);
\draw (0,-1) to (0.5,-0.5);
\draw (0.5,-1) arc (-90:90:0.25);
\draw (0,-1) [->]  to (0,-1.5);
\draw (0.6,-1.25) node {$(t, N)$};
\draw (-0.5,0.2) node {$t^{-1}$};
\end{tikzpicture}\right)=(t^{2}-t^{-2})
\left(\begin{tikzpicture}[baseline=-0.65ex, thick, scale=1]
\draw (0.5,-1) [<-]  to (0.5,1.2);
\draw (0.5,-1.25) node {$(t, N)$};
\end{tikzpicture}\right).
\end{align*}
For the component where the weight is hidden, we can choose any integer as its weight. We prove one relation, and the other three can be proved in the same vein. From Table 1, we see that the contribution of a negative full-twist is $t^{2N}$ if the corresponding edge has color $(t, N)$. We have
\begin{align*}
\left(\begin{tikzpicture}[baseline=-0.65ex, thick, scale=1]
\draw (0,-0.5) to (0,0.9);
\draw (0,1.1) to (0,1.4);
\draw (-0.3,1) [->]to (0.3,1);
\draw (-0.3,0.5) to (-0.1,0.5);
\draw (0.1,0.5) to (0.3,0.5);
\draw  (-0.3,1) arc (90:270:0.25);
\draw (0.3,1) to (0.8,0.5);
\draw (0.3,0.5) to (0.5,0.7);
\draw (0.8,1) to (0.6,0.8);
\draw (0,-0.5) to (0.5,-1);
\draw (0.8,0.5) arc (-90:90:0.25);
\draw (0,-1) to (0.2,-0.8);
\draw (0.3,-0.7) to (0.5,-0.5);
\draw (0.5,-1) arc (-90:90:0.25);
\draw (0,-1) [->]  to (0,-1.5);
\draw (0.6,-1.25) node {$(t, N)$};
\draw (-1,0.2) node {$(t^{-1}, N')$};
\end{tikzpicture}\right)=t^{2(N-N')}\left(\begin{tikzpicture}[baseline=-0.65ex, thick, scale=1]
\draw (0,1.1) to (0,1.4);
\draw (-0.3,1) [->]to (0.3,1);
\draw (-0.3,0.5) to (-0.1,0.5);
\draw (0.1,0.5) to (0.3,0.5);
\draw  (-0.3,1) arc (90:270:0.25);
\draw  (0.3,1) arc (90:-90:0.25);
\draw (0,0.9) [->]  to (0,-1);
\draw (0.6,-0.8) node {$(t, N)$};
\draw (-0.8,0.2) node {$(t^{-1}, N')$};
\end{tikzpicture}\right)=t^{2(N-N')}\left(\begin{tikzpicture}[baseline=-0.65ex, thick, scale=1]
\draw (-0.3,1) [->-]to (0.3,0.4);
\draw (-0.3,0.4) to (-0.1,0.6);
\draw (0.1,0.8) to (0.3,1);
\draw (-0.3,0.4) to (0.3,-0.2);
\draw (-0.3,-0.2) to (-0.1,0);
\draw (0.1,0.2) to (0.3,0.4);
\draw  (-0.3,1) arc (90:270:0.6);
\draw (0.3,-0.2) [->]  to (0.3,-1);
\draw (0.3,1) to (0.3,1.2);
\draw (0.9,-0.8) node {$(t, N)$};
\draw (-0.6,-0.5) node {$(t^{-1}, N')$};
\end{tikzpicture}\right)\\
=-t^{2(N-N')}t^{2(N'-N)}(t^2-t^{-2})\left(\begin{tikzpicture}[baseline=-0.65ex, thick, scale=1]
\draw (0.5,-1) [<-]  to (0.5,1.2);
\draw (0.6,-1.25) node {$(t, N)$};
\end{tikzpicture}\right)=-(t^2-t^{-2})\left(\begin{tikzpicture}[baseline=-0.65ex, thick, scale=1]
\draw (0.5,-1) [<-]  to (0.5,1.2);
\draw (0.6,-1.25) node {$(t, N)$};
\end{tikzpicture}\right).
\end{align*}
The third equality follows from the calculation as we did in Example~\ref{example} (with an overall change of orientations of edges).
\end{proof}

Next, we study how the Alexander polynomial changes under a handle-slide move. We have the following lemma.

\begin{lemma}
\label{handle-slide}
Suppose $\Lambda$ is a colored framed graph, and $K$ is a knot component of $\Lambda$ with Kirby color $\Omega (s, S)$. Let $e$ be an oriented edge of $\Lambda\backslash K$ with color $(t, N)$. Let $\Lambda'$ be a graph obtained from $\Lambda$ by a handle-slide move of $e$ along $K$, and $K$ has the new Kirby color $\Lambda (ts, N+S-1)$. If $clk(K, \Lambda)=1$ and $(ts)^4\neq 1$, we have $\Delta (\Lambda)=\Delta (\Lambda')$.
\end{lemma}
\begin{proof} We have
\begin{eqnarray*}
&&\left(\begin{tikzpicture}[baseline=-0.65ex, thick, scale=0.8]
\draw (0.5,-1.2) [<-]  to (0.5,1.4);
\draw (0.5,-1.5) node {$\scriptstyle (t, N)$};
\draw (0.2,1.2) node {$\scriptstyle e$};
\draw [dotted](2.2,-0.8) arc (-90:90:1);
\draw (2.2,-0.8) [<-]arc (-90:-270:1);
\draw (2.2,-1.2) node {$\scriptstyle \Omega(s, S)$};
\draw (1.5,1.2) node {$\scriptstyle K$};
\end{tikzpicture}\right)=d(s)\left(\begin{tikzpicture}[baseline=-0.65ex, thick, scale=0.8]
\draw (0.5,-1.2) [<-]  to (0.5,1.4);
\draw (0.5,-1.5) node {$\scriptstyle (t, N)$};
\draw [dotted](2.2,-0.8) arc (-90:90:1);
\draw (2.2,-0.8) [<-]arc (-90:-270:1);
\draw (2.2,-1.2) node {$\scriptstyle (s, S)$};
\end{tikzpicture}\right)-d(s)\left(\begin{tikzpicture}[baseline=-0.65ex, thick, scale=0.8]
\draw (0.5,-1.2) [<-]  to (0.5,1.4);
\draw (0.5,-1.5) node {$\scriptstyle (t, N)$};
\draw [dotted](2.2,-0.8) arc (-90:90:1);
\draw (2.2,-0.8) [->]arc (-90:-270:1);
\draw (2.2,-1.2) node {$\scriptstyle (s^{-1}, 2-S)$};
\end{tikzpicture}\right)\\
&=&d(s)d(ts)\left(\begin{tikzpicture}[baseline=-0.65ex, thick, scale=1]
\draw (0,-1) [<-] to  (0.5,-0.5);
\draw (0.5, -0.5)  [->]  to  (1,-1);
\draw (0.5, 0.5)    [-<-] to (0,1);
\draw (0.5,0.5) [-<-] to  (1,1);
\draw (0.5,-0.5) [-<-] to  (0.5,0.5);
\draw (0, -1.25) node {$\scriptstyle (t, N)$};
\draw (1.5, 0) node {${ \scriptstyle (ts, N+S-1)}$};
\draw (0, 1.25) node {$\scriptstyle (t, N)$};
\draw (1, 1.25) node {$\scriptstyle (s, S)$};
\draw (1,1)  to  (1.5,1);
\draw (1,-1)  to  (1.5,-1);
\draw [dotted](1.5,-1) arc (-90:90:1);
\end{tikzpicture}\right)-d(s)d(ts)\left(\begin{tikzpicture}[baseline=-0.65ex, thick, scale=1]
\draw (0,-1) [<-] to  (0.5,-0.5);
\draw (0.5, -0.5)  [->]  to  (1,-1);
\draw (0.5, 0.5)    [-<-] to (0,1);
\draw (0.5,0.5) [-<-] to  (1,1);
\draw (0.5,-0.5) [->-] to  (0.5,0.5);
\draw (0, -1.25) node {$\scriptstyle (t, N)$};
\draw (1.8, 0) node {${ \scriptstyle ((ts)^{-1}, -N-S-1)}$};
\draw (0, 1.25) node {$\scriptstyle (t, N)$};
\draw (1, 1.25) node {$\scriptstyle (s, S)$};
\draw (1,1)  to  (1.5,1);
\draw (1,-1)  to  (1.5,-1);
\draw [dotted](1.5,-1) arc (-90:90:1);
\end{tikzpicture}\right)\\
&&-d(s)d(ts)\left(\begin{tikzpicture}[baseline=-0.65ex, thick, scale=1]
\draw (0,-1) [<-] to  (0.5,-0.5);
\draw (0.5, -0.5)  [-<-]  to  (1,-1);
\draw (0.5, 0.5)    [-<-] to (0,1);
\draw (0.5,0.5) [->] to  (1,1);
\draw (0.5,-0.5) [-<-] to  (0.5,0.5);
\draw (0, -1.25) node {$\scriptstyle (t, N)$};
\draw (1.5, 0) node {${ \scriptstyle (ts, N+S-1)}$};
\draw (0, 1.25) node {$\scriptstyle (t, N)$};
\draw (1.3, 1.25) node {$\scriptstyle (s^{-1}, 2-S)$};
\draw (1,1)  to  (1.5,1);
\draw (1,-1)  to  (1.5,-1);
\draw [dotted](1.5,-1) arc (-90:90:1);
\end{tikzpicture}\right)+d(s)d(ts)\left(\begin{tikzpicture}[baseline=-0.65ex, thick, scale=1]
\draw (0,-1) [<-] to  (0.5,-0.5);
\draw (0.5, -0.5)  [-<-]  to  (1,-1);
\draw (0.5, 0.5)    [-<-] to (0,1);
\draw (0.5,0.5) [->] to  (1,1);
\draw (0.5,-0.5) [->-] to  (0.5,0.5);
\draw (0, -1.25) node {$\scriptstyle (t, N)$};
\draw (1.7, 0) node {${ \scriptstyle ((ts)^{-1}, 3-N-S)}$};
\draw (0, 1.25) node {$\scriptstyle (t, N)$};
\draw (1.3, 1.25) node {$\scriptstyle (s^{-1}, 2-S)$};
\draw (1,1)  to  (1.5,1);
\draw (1,-1)  to  (1.5,-1);
\draw [dotted](1.5,-1) arc (-90:90:1);
\end{tikzpicture}\right)\\
&=&d(s)d(ts)\left(\begin{tikzpicture}[baseline=-0.65ex, thick, scale=1]
\draw (0,-1) [<-] to  (0.5,-0.5);
\draw (0.5, -0.5)  [->]  to  (1,-1);
\draw (0.5, 0.5)    [-<-] to (0,1);
\draw (0.5,0.5) [-<-] to  (1,1);
\draw (0.5,-0.5) [-<-] to  (0.5,0.5);
\draw (0, -1.25) node {$\scriptstyle (t, N)$};
\draw (1.5, 0) node {${ \scriptstyle (ts, N+S-1)}$};
\draw (0, 1.25) node {$\scriptstyle (t, N)$};
\draw (1, 1.25) node {$\scriptstyle (s, S)$};
\draw (1,1)  to  (1.5,1);
\draw (1,-1)  to  (1.5,-1);
\draw [dotted](1.5,-1) arc (-90:90:1);
\end{tikzpicture}\right)-d(s)d(ts)\left(\begin{tikzpicture}[baseline=-0.65ex, thick, scale=1]
\draw (0,-1) [<-] to  (0.5,-0.5);
\draw (0.5, -0.5)  [->]  to  (1,-1);
\draw (0.5, 0.5)    [-<-] to (0,1);
\draw (0.5,0.5) [-<-] to  (1,1);
\draw (0.5,-0.5) [->-] to  (0.5,0.5);
\draw (0, -1.25) node {$\scriptstyle (t, N)$};
\draw (1.7, 0) node {${ \scriptstyle ((ts)^{-1}, 3-N-S)}$};
\draw (0, 1.25) node {$\scriptstyle (t, N)$};
\draw (1.2, 1.2) node {$\scriptstyle (s, S-4)$};
\draw (1,1)  to  (1.5,1);
\draw (1,-1)  to  (1.5,-1);
\draw [dotted](1.5,-1) arc (-90:90:1);
\end{tikzpicture}\right)\\
&&-d(s)d(ts)\left(\begin{tikzpicture}[baseline=-0.65ex, thick, scale=1]
\draw (0,-1) [<-] to  (0.5,-0.5);
\draw (0.5, -0.5)  [-<-]  to  (1,-1);
\draw (0.5, 0.5)    [-<-] to (0,1);
\draw (0.5,0.5) [->] to  (1,1);
\draw (0.5,-0.5) [-<-] to  (0.5,0.5);
\draw (0, -1.25) node {$\scriptstyle (t, N)$};
\draw (1.5, 0) node {${ \scriptstyle (ts, N+S-1)}$};
\draw (0, 1.25) node {$\scriptstyle (t, N)$};
\draw (1.3, 1.25) node {$\scriptstyle (s^{-1}, 2-S)$};
\draw (1,1)  to  (1.5,1);
\draw (1,-1)  to  (1.5,-1);
\draw [dotted](1.5,-1) arc (-90:90:1);
\end{tikzpicture}\right)+d(s)d(ts)\left(\begin{tikzpicture}[baseline=-0.65ex, thick, scale=1]
\draw (0,-1) [<-] to  (0.5,-0.5);
\draw (0.5, -0.5)  [-<-]  to  (1,-1);
\draw (0.5, 0.5)    [-<-] to (0,1);
\draw (0.5,0.5) [->] to  (1,1);
\draw (0.5,-0.5) [->-] to  (0.5,0.5);
\draw (0, -1.25) node {$\scriptstyle (t, N)$};
\draw (1.7, 0) node {${ \scriptstyle ((ts)^{-1}, 3-N-S)}$};
\draw (0, 1.25) node {$\scriptstyle (t, N)$};
\draw (1.2, 1.25) node {$\scriptstyle (s^{-1}, 2-S)$};
\draw (1,1)  to  (1.5,1);
\draw (1,-1)  to  (1.5,-1);
\draw [dotted](1.5,-1) arc (-90:90:1);
\end{tikzpicture}\right)\\
&=&d(s)d(ts)\left(\begin{tikzpicture}[baseline=-0.65ex, thick, scale=1]
\draw (0,0)  [->-]to  (0,-0.5);
\draw (0,0) [-<-]to (0.3, 0.3);
\draw (0,0) arc (0:180:0.3);
\draw (0,-0.5) to (0.3, -0.8);
\draw (0,-0.5) arc (90:270:0.3);
\draw [dotted] (0.3, -0.8) arc (-90:90:0.55);
\draw [dotted] (0.3, -1.1) arc (-90:90:0.9);
\draw (0, -1.1) to (0.3, -1.1);
\draw (-0.6,0) to (-0.6, -0.5);
\draw (-0.6, -0.5)[->] to (-1, -1);
\draw (0.3,0.7) [<-]to (-0.6, 0.7);
\draw (-0.6, 0.7) to (-0.9, 1);
\draw (-1, -1.25) node {${\scriptstyle (t, N)}$};
\draw (-0.3, 0.9) node {${\scriptstyle (t, N)}$};
\draw (0.35, -0.2) node {${\scriptstyle (s, S)}$};
\draw (1, -1.3) node {${ \scriptstyle (ts, N+S-1)}$};
\draw [thin](1.2, -1.1) to (0.8, -0.6);
\end{tikzpicture}\right)-d(s)d(ts) \left(\begin{tikzpicture}[baseline=-0.65ex, thick, scale=1]
\draw (0,0)  [->-]to  (0,-0.5);
\draw (0,0) [->]to (0.3, 0.3);
\draw (0,0) arc (0:180:0.3);
\draw (0,-0.5) to (0.3, -0.8);
\draw (0,-0.5) arc (90:270:0.3);
\draw [dotted] (0.3, -0.8) arc (-90:90:0.55);
\draw [dotted] (0.3, -1.1) arc (-90:90:0.9);
\draw (0, -1.1) to (0.3, -1.1);
\draw (-0.6,0) to (-0.6, -0.5);
\draw (-0.6, -0.5)[->] to (-1, -1);
\draw (0.3,0.7) [<-]to (-0.6, 0.7);
\draw (-0.6, 0.7) to (-0.9, 1);
\draw (-1, -1.25) node {${\scriptstyle (t, N)}$};
\draw (-0.3, 0.9) node {${\scriptstyle (t, N)}$};
\draw (0.55, -0.2) node {${\scriptstyle (s, S-4)}$};
\draw (1, -1.3) node {${ \scriptstyle ((ts)^{-1}, 3-N-S)}$};
\draw [thin](1.2, -1.1) to (0.8, -0.6);
\end{tikzpicture}\right)\\
&&-d(s)d(ts)\left(\begin{tikzpicture}[baseline=-0.65ex, thick, scale=1]
\draw (0,0)  [-<-]to  (0,-0.5);
\draw (0,0) [-<-]to (0.3, 0.3);
\draw (0,0) arc (0:180:0.3);
\draw (0,-0.5) to (0.3, -0.8);
\draw (0,-0.5) arc (90:270:0.3);
\draw [dotted] (0.3, -0.8) arc (-90:90:0.55);
\draw [dotted] (0.3, -1.1) arc (-90:90:0.9);
\draw (0, -1.1) to (0.3, -1.1);
\draw (-0.6,0) to (-0.6, -0.5);
\draw (-0.6, -0.5)[->] to (-1, -1);
\draw (0.3,0.7) [<-]to (-0.6, 0.7);
\draw (-0.6, 0.7) to (-0.9, 1);
\draw (-1, -1.25) node {${\scriptstyle (t, N)}$};
\draw (-0.3, 0.9) node {${\scriptstyle (t, N)}$};
\draw (0.8, -0.2) node {${\scriptstyle (s^{-1}, 2-S)}$};
\draw (1, -1.3) node {${ \scriptstyle (ts, N+S-1)}$};
\draw [thin](1.2, -1.1) to (0.8, -0.6);
\end{tikzpicture}\right)+d(s)d(ts) \left(\begin{tikzpicture}[baseline=-0.65ex, thick, scale=1]
\draw (0,0)  [-<-]to  (0,-0.5);
\draw (0,0) [->]to (0.3, 0.3);
\draw (0,0) arc (0:180:0.3);
\draw (0,-0.5) to (0.3, -0.8);
\draw (0,-0.5) arc (90:270:0.3);
\draw [dotted] (0.3, -0.8) arc (-90:90:0.55);
\draw [dotted] (0.3, -1.1) arc (-90:90:0.9);
\draw (0, -1.1) to (0.3, -1.1);
\draw (-0.6,0) to (-0.6, -0.5);
\draw (-0.6, -0.5)[->] to (-1, -1);
\draw (0.3,0.7) [<-]to (-0.6, 0.7);
\draw (-0.6, 0.7) to (-0.9, 1);
\draw (-1, -1.25) node {${\scriptstyle (t, N)}$};
\draw (-0.3, 0.9) node {${\scriptstyle (t, N)}$};
\draw (0.8, -0.2) node {${\scriptstyle (s^{-1}, 2-S)}$};
\draw (1, -1.3) node {${ \scriptstyle ((ts)^{-1}, 3-N-S)}$};
\draw [thin](1.2, -1.1) to (0.8, -0.6);
\end{tikzpicture}\right)\\
&=&d(ts)\left [d(s) 
\left(\begin{tikzpicture}[baseline=-0.65ex, thick, scale=1]
\draw (0,0)  [->-]to  (0,-0.5);
\draw (0,0) [-<-]to (0.3, 0.3);
\draw (0,0) arc (0:180:0.3);
\draw (0,-0.5) to (0.3, -0.8);
\draw (0,-0.5) arc (90:270:0.3);
\draw [dotted] (0.3, -0.8) arc (-90:90:0.55);
\draw [dotted] (0.3, -1.1) arc (-90:90:0.9);
\draw (0, -1.1) to (0.3, -1.1);
\draw (-0.6,0) to (-0.6, -0.5);
\draw (-0.6, -0.5)[->] to (-1, -1);
\draw (0.3,0.7) [<-]to (-0.6, 0.7);
\draw (-0.6, 0.7) to (-0.9, 1);
\draw (-1, -1.25) node {${\scriptstyle (t, N)}$};
\draw (-0.3, 0.9) node {${\scriptstyle (t, N)}$};
\draw (0.35, -0.2) node {${\scriptstyle (s, S)}$};
\draw (1, -1.3) node {${ \scriptstyle (ts, N+S-1)}$};
\draw [thin](1.2, -1.1) to (0.8, -0.6);
\end{tikzpicture}\right)-d(s)\left(\begin{tikzpicture}[baseline=-0.65ex, thick, scale=1]
\draw (0,0)  [-<-]to  (0,-0.5);
\draw (0,0) [-<-]to (0.3, 0.3);
\draw (0,0) arc (0:180:0.3);
\draw (0,-0.5) to (0.3, -0.8);
\draw (0,-0.5) arc (90:270:0.3);
\draw [dotted] (0.3, -0.8) arc (-90:90:0.55);
\draw [dotted] (0.3, -1.1) arc (-90:90:0.9);
\draw (0, -1.1) to (0.3, -1.1);
\draw (-0.6,0) to (-0.6, -0.5);
\draw (-0.6, -0.5)[->] to (-1, -1);
\draw (0.3,0.7) [<-]to (-0.6, 0.7);
\draw (-0.6, 0.7) to (-0.9, 1);
\draw (-1, -1.25) node {${\scriptstyle (t, N)}$};
\draw (-0.3, 0.9) node {${\scriptstyle (t, N)}$};
\draw (0.8, -0.2) node {${\scriptstyle (s^{-1}, 2-S)}$};
\draw (1, -1.3) node {${ \scriptstyle (ts, N+S-1)}$};
\draw [thin](1.2, -1.1) to (0.8, -0.6);
\end{tikzpicture}\right)
\right]\\
&&-d(ts) \left[d(s)
\left(\begin{tikzpicture}[baseline=-0.65ex, thick, scale=1]
\draw (0,0)  [->-]to  (0,-0.5);
\draw (0,0) [->]to (0.3, 0.3);
\draw (0,0) arc (0:180:0.3);
\draw (0,-0.5) to (0.3, -0.8);
\draw (0,-0.5) arc (90:270:0.3);
\draw [dotted] (0.3, -0.8) arc (-90:90:0.55);
\draw [dotted] (0.3, -1.1) arc (-90:90:0.9);
\draw (0, -1.1) to (0.3, -1.1);
\draw (-0.6,0) to (-0.6, -0.5);
\draw (-0.6, -0.5)[->] to (-1, -1);
\draw (0.3,0.7) [<-]to (-0.6, 0.7);
\draw (-0.6, 0.7) to (-0.9, 1);
\draw (-1, -1.25) node {${\scriptstyle (t, N)}$};
\draw (-0.3, 0.9) node {${\scriptstyle (t, N)}$};
\draw (0.55, -0.2) node {${\scriptstyle (s, S-4)}$};
\draw (1, -1.3) node {${ \scriptstyle ((ts)^{-1}, 3-N-S)}$};
\draw [thin](1.2, -1.1) to (0.8, -0.6);
\end{tikzpicture}\right)-d(s)\left(\begin{tikzpicture}[baseline=-0.65ex, thick, scale=1]
\draw (0,0)  [-<-]to  (0,-0.5);
\draw (0,0) [->]to (0.3, 0.3);
\draw (0,0) arc (0:180:0.3);
\draw (0,-0.5) to (0.3, -0.8);
\draw (0,-0.5) arc (90:270:0.3);
\draw [dotted] (0.3, -0.8) arc (-90:90:0.55);
\draw [dotted] (0.3, -1.1) arc (-90:90:0.9);
\draw (0, -1.1) to (0.3, -1.1);
\draw (-0.6,0) to (-0.6, -0.5);
\draw (-0.6, -0.5)[->] to (-1, -1);
\draw (0.3,0.7) [<-]to (-0.6, 0.7);
\draw (-0.6, 0.7) to (-0.9, 1);
\draw (-1, -1.25) node {${\scriptstyle (t, N)}$};
\draw (-0.3, 0.9) node {${\scriptstyle (t, N)}$};
\draw (0.75, -0.2) node {${\scriptstyle (s^{-1}, 2-S)}$};
\draw (1, -1.3) node {${ \scriptstyle ((ts)^{-1}, 3-N-S)}$};
\draw [thin](1.2, -1.1) to (0.8, -0.6);
\end{tikzpicture}\right)
\right]\\
&=&d(ts)\left(\begin{tikzpicture}[baseline=-0.65ex, thick, scale=1]
\draw [dotted] (0.3, -0.8) arc (-90:90:0.55);
\draw [dotted] (0.3, -1.1) arc (-90:90:0.9);
\draw [->] (0, -1.1) to [out=180,in=270]  (-0.3,-0.2) to [out=90,in=0] (-0.7,0.3) to [out=180,in=90] (-1.1, -0.2) to [out=270,in=45] (-1.4, -1.1);
\draw [-<-](0.3, -0.8) to [out=180, in=270] (0, -0.25) to [out=90, in=180] (0.3, 0.3);
\draw (0, -1.1) to (0.3, -1.1);
\draw (0.3,0.7) [<-]to (-0.6, 0.7);
\draw (-0.6, 0.7) to (-0.9, 1);
\draw (-1, -1.25) node {${\scriptstyle (t, N)}$};
\draw (-0.3, 0.9) node {${\scriptstyle (t, N)}$};
\draw (1, -1.3) node {${ \scriptstyle (ts, N+S-1)}$};
\draw [thin](1.2, -1.1) to (0.8, -0.6);
\end{tikzpicture}\right)-d(ts)\left(\begin{tikzpicture}[baseline=-0.65ex, thick, scale=1]
\draw [dotted] (0.3, -0.8) arc (-90:90:0.55);
\draw [dotted] (0.3, -1.1) arc (-90:90:0.9);
\draw [->] (0, -1.1) to [out=180,in=270]  (-0.3,-0.2) to [out=90,in=0] (-0.7,0.3) to [out=180,in=90] (-1.1, -0.2) to [out=270,in=45] (-1.4, -1.1);
\draw [->-] (0.3, -0.8) to [out=180, in=270] (0, -0.25) to [out=90, in=180] (0.3, 0.3);
\draw (0, -1.1) to (0.3, -1.1);
\draw (0.3,0.7) [<-]to (-0.6, 0.7);
\draw (-0.6, 0.7) to (-0.9, 1);
\draw (-1, -1.25) node {${\scriptstyle (t, N)}$};
\draw (-0.3, 0.9) node {${\scriptstyle (t, N)}$};
\draw (1, -1.3) node {${ \scriptstyle ((ts)^{-1}, 3-N-S)}$};
\draw [thin](1.2, -1.1) to (0.8, -0.6);
\end{tikzpicture}\right)\\
&=&\left(\begin{tikzpicture}[baseline=-0.65ex, thick, scale=1]
\draw [dotted] (0.3, -0.8) arc (-90:90:0.55);
\draw [dotted] (0.3, -1.1) arc (-90:90:0.9);
\draw [->] (0, -1.1) to (-1.4, -1.1);
\draw [-<-](0.3, -0.8) to [out=180, in=270] (0, -0.25) to [out=90, in=180] (0.3, 0.3);
\draw (0, -1.1) to (0.3, -1.1);
\draw (0.3,0.7) [<-]to (-0.6, 0.7);
\draw (-0.6, 0.7) to (-0.9, 1);
\draw (-0.3, 0.9) node {${\scriptstyle (t, N)}$};
\draw (-1, -0.1) node {${ \scriptstyle \Omega(ts, N+S-1)}$};
\end{tikzpicture}\right).
\end{eqnarray*}
The second equality follows from Lemma \ref{junction}, the junction relations. To obtain the third equality, we need the following fact:
\begin{eqnarray*}
\left(\begin{tikzpicture}[baseline=-0.65ex, thick, scale=1]
\draw (0,-1) [<-] to  (0.5,-0.5);
\draw (0.5, -0.5)  [->]  to  (1,-1);
\draw (0.5, 0.5)    [-<-] to (0,1);
\draw (0.5,0.5) [-<-] to  (1,1);
\draw (0.5,-0.5) [->-] to  (0.5,0.5);
\draw (0, -1.25) node {$\scriptstyle (t, N)$};
\draw (1.8, 0) node {${ \scriptstyle ((ts)^{-1}, -N-S-1)}$};
\draw (0, 1.25) node {$\scriptstyle (t, N)$};
\draw (1.1, 1.25) node {$\scriptstyle (s, S)$};
\draw (1,1)  to  (1.5,1);
\draw (1,-1)  to  (1.5,-1);
\draw [dotted](1.5,-1) arc (-90:90:1);
\end{tikzpicture}\right)
=\left(\begin{tikzpicture}[baseline=-0.65ex, thick, scale=1]
\draw (0,-1) [<-] to  (0.5,-0.5);
\draw (0.5, -0.5)  [->]  to  (1,-1);
\draw (0.5, 0.5)    [-<-] to (0,1);
\draw (0.5,0.5) [-<-] to  (1,1);
\draw (0.5,-0.5) [->-] to  (0.5,0.5);
\draw (0, -1.25) node {$\scriptstyle (t, N)$};
\draw (1.7, 0) node {${ \scriptstyle ((ts)^{-1}, 3-N-S)}$};
\draw (0, 1.25) node {$\scriptstyle (t, N)$};
\draw (1.1, 1.25) node {$\scriptstyle (s, S-4)$};
\draw (1,1)  to  (1.5,1);
\draw (1,-1)  to  (1.5,-1);
\draw [dotted](1.5,-1) arc (-90:90:1);
\end{tikzpicture}\right),
\end{eqnarray*} 
which holds because of the assumption that $clk(K, \Lambda)=1$. The edge colored by $(s, S)$ was the main part of $K$ before junction, so it has the same crossings with $\Lambda$ as $K$. Following the proof of Lemma \ref{zero}, we see that the Alexander polynomial is invariant if we change the color $(s, S)$ to $(s, S-4)$. By admissibility condition, the color $((ts)^{-1}, -N-S-1)$ should be changed to $((ts)^{-1}, 3-N-S)$. Note that the Boltzmann weights around a vertex do not depend on weights. 
The forth equality holds because the each diagram after the equality are obtained from a previous diagram by sliding of the edge with color $(t, T)$ along the edge coming from $K$, which is an isotopy of the graph. The sixth equality follows again from the junction relations.
\end{proof}

\medskip

\section{Invariant for 3-manifolds}
In this section, we consider a $1$-palette for which $G$ is a finitely generated abelian group containing at least one $\mathbb{Z}$ summand and satisfying $t^4=1 \iff t=1$. Namely $G$ contains $\mathbb{Z}$ but no $\mathbb{Z}/2\mathbb{Z}$ as a subgroup. It is not hard to find such $1$-palettes as we can see in the following examples.

\begin{ex}
\label{palette}
The $1$-palettes defined by the following data meet our requirements.
\begin{enumerate}
\item Let $B=\mathbb{Q}(t)$, the field of rational functions of $t$, and let $G=\mathbb{Z}\langle t \rangle$, the cyclic group generated by $t$.
\item Let $\xi_l$ be the $l$-th primitive root of unity for a prime number $l\geq 3$. Let $B=\mathbb{Q}(\pi, \xi_l)$, the extension field of $\mathbb{Q}$ generated by $\pi$ and $\xi_l$. Let $G=\mathbb{Z}\langle \pi, \xi_l \rangle$, the abelian group generated by $\pi$ and $\xi_l$. 
\end{enumerate}
\end{ex}

Let $M$ be a 3-manifold, and let $\Gamma$ be a colored framed graph in $M$ colored by a $1$-palette $(B, G)$ where $G$ contains $\mathbb{Z}$ but no $\mathbb{Z}/2\mathbb{Z}$ as a subgroup. Consider a cohomology class $\omega: H_{1}(M\backslash \Gamma, \mathbb{Z})\to G$. We say that $(M, \Gamma, \omega)$ is a {\it compatible triple} if for each edge $e$ of $\Gamma$, the multiplicity of $e$ is equal to $\omega([m_e])$, where $m_e$ is the meridian of $e$. Let $L$ be a surgery presentation of $M$. We say that $L$ is {\it computable} for a compatible triple $(M, \Gamma, \omega)$ if $L\cup \Gamma\neq\emptyset$ and $\omega ([m])\neq 1\in G$ for any meridian $m$ of $L$. We show the existence of the computable surgery presentations.

\begin{lemma}
For a compatible triple $(M, \Gamma, \omega)$ over the $1$-palette $(B, G)$ where $G$ contains $\mathbb{Z}$ but no $\mathbb{Z}/2\mathbb{Z}$ as a subgroup, if $\omega$ is non-trivial, then there exists a computable surgery presentation of $(M, \Gamma, \omega)$.
\end{lemma}
\begin{proof}
Choose a surgery presentation of $M$, which we call $L$. Recall that $H_{1}(M\backslash \Gamma, \mathbb{Z})$ is generated by the meridians of $L\cup \Gamma$. Since $\omega$ is non-trivial, there is an edge or component $e$ of $L\cup \Gamma$ for which $\omega ([m_e])\neq 1$, where $m_e$ is the meridian of $e$. We do a blow-up move on the edge $e$ to create a component $K$. Then $\omega([m_K])=\omega ([m_e])\neq 1$, and $K\cup L$ is a new surgery presentation of $M$, where $m_K$ is the meridian of $K$.
If $K\cup L$ is not computable, we can slide $K$ along those components of $L$ whose meridians are mapped to $1$ under $w$ to get a computable surgery presentation. Precisely, if $L_0$ is a component of $L$ with $\omega ([m_{L_0}])=1$, we slide $K$ along $L_0$ to get a new surgery presentation. For this new presentation, $\omega ([m_{L_0}])=\omega([m_K])\neq 1$, where $m_{L_0}$ is the meridian of $L_0$. 
\end{proof}

Now we are ready to prove our main result. The proof is essentially the same as the proof of Theorem 4.7 in \cite{MR3286896}. Since our situation is concrete, we state the proof for the completeness of the paper.

\begin{proof}[Proof of Theorem \ref{mainresult1}]
Let $L$ and $L'$ be two computable surgery presentations of $(M, \Gamma, \omega)$. By Theorem \ref{kirby}, there is a sequence of handle-slide moves, blow-up/down moves connecting $L\cup \Gamma$ and $L'\cup \Gamma$ and the induced diffeomorphism $f: M\to M$ satisfies $f(\Gamma)=\Gamma$ and $f^{*}(\omega)=\omega$. We want to show that $\frac{\Delta (L\cup \Gamma)}{2^{r(L)}(-1)^{\sigma_+(L)}}=\frac{\Delta (L'\cup \Gamma)}{2^{r(L')}(-1)^{\sigma_+(L')}}$.

We can assume that all the blow-up moves are at the beginning and all the blow-down moves are at the end. Namely it is a sequence as follows:
$$L=L_0\to L_1\to \cdots \to L_k \to L_{k+1}\to \cdots \to L_{k+l}\to L_{k+l+1}\to \cdots \to L_{k+l+m}=L',$$
where $L_0$ and $L_{k}$ are connected by blow-up moves, $L_{k+l}$ and $L_{k+l+m}$ are connected by blow-down moves, while $L_{k}$ and $L_{k+l}$ are connected by handle-slide moves.

If $k\geq 1$, $L_1\cup \Gamma$ is obtained from $L_0\cup \Gamma$ by a blow-up move of a component/edge of $L_0\cup \Gamma$, which is still computable since the circle component created by the blow-up move has the same multiplicity as that of the edge where the move is done. By Lemma~\ref{up-down}, we have $\frac{\Delta (L_0\cup \Gamma)}{2^{r(L_0)}(-1)^{\sigma_+(L_0)}}=\frac{\Delta (L_1\cup \Gamma)}{2^{r(L_1)}(-1)^{\sigma_+(L_1)}}$, since the change of numerator is annihilated by the change of denominator. Subsequently, one can show that 
$\frac{\Delta (L_0\cup \Gamma)}{2^{r(L_0)}(-1)^{\sigma_+(L_0)}}=\frac{\Delta (L_k\cup \Gamma)}{2^{r(L_k)}(-1)^{\sigma_+(L_k)}}$.

Since $L'\cup \Gamma$ is obtained from $L_{k+l}\cup \Gamma$ by blow-down moves, so conversely $L_{k+l}\cup \Gamma$ is obtained from $L'\cup \Gamma$ by blow-up moves. As we discussion above, in this case we have $\frac{\Delta (L'\cup \Gamma)}{2^{r(L')}(-1)^{\sigma_+(L')}}=\frac{\Delta (L_{k+l}\cup \Gamma)}{2^{r(L_{k+l})}(-1)^{\sigma_+(L_{k+l})}}$.

Now it suffices to show that $\frac{\Delta (L_k\cup \Gamma)}{2^{r(L_k)}(-1)^{\sigma_+(L_k)}}=\frac{\Delta (L_{k+l}\cup \Gamma)}{2^{r(L_{k+l})}(-1)^{\sigma_+(L_{k+l})}}$, where $L_k\cup \Gamma$ and $L_{k+l}\cup \Gamma$ are computable and are connected by handle-slide moves. More precisely, let
\begin{equation*}
\label{sequence}
L_k \xrightarrow{s_1} L_{k+1}\xrightarrow{s_2} \cdots \xrightarrow{s_l} L_{k+l} \tag{H1}
\end{equation*}
be such a sequence of handle-slide moves.

Suppose $L_{k+1}$ is computable, which is obtained from $L_k$ by a handle-slide move of $e\subset L_k\cup \Gamma$ along a component $K\subset L_k$. Note that 
\begin{eqnarray*}
\displaystyle clk(K, L_k\cup \Gamma)&=&\prod_{\text{$e$: edge of $L_k\cup \Gamma$}} t_e^{lk(K, e)}=\prod_{\text{$e$: edge of $L_k\cup \Gamma$}} \omega ([m_e])^{lk(K, e)}\\&=&\omega\left (\sum_{\text{$e$: edge of $L_k\cup \Gamma$}} lk(K, e) [m_e]\right )=\omega (0)=1,
\end{eqnarray*}
where the forth equality follows from the presentation of $H_1(M\backslash \Gamma, \mathbb{Z})$. By Lemma~\ref{handle-slide} $\Delta (L_k\cup \Gamma)=\Delta (L_{k+1}\cup \Gamma)$, and thus $\frac{\Delta (L_k\cup \Gamma)}{2^{r(L_k)}(-1)^{\sigma_+(L_k)}}=\frac{\Delta (L_{k+1}\cup \Gamma)}{2^{r(L_{k+1})}(-1)^{\sigma_+(L_{k+1})}}$, since the component number and the eigenvalues of the linking matrix do not change under a handle-slide move.
If the intermediate presentations between $L_k$ and $L_{k+l}$ are all computable, we have $\frac{\Delta (L_k\cup \Gamma)}{2^{r(L_k)}(-1)^{\sigma_+(L_k)}}=\frac{\Delta (L_{k+l}\cup \Gamma)}{2^{r(L_{k+l})}(-1)^{\sigma_+(L_{k+l})}}$.

Now we consider the case that some surgery presentations between $L_k$ and $L_{k+l}$ are not computable. Namely some knot components have multiplicity $1$. We separate the discussion to the cases that $\Gamma\neq \emptyset$ and $\Gamma= \emptyset$ (and thus $L_k\neq \emptyset$ since $L_k\cup \Gamma\neq \emptyset$).

Suppose $\Gamma\neq \emptyset$. We choose an edge $e\subset \Gamma$ with color $(\beta, N)$ and let $\tilde{\Gamma}=\Gamma\cup \{m_e\}$, where $m_e$ is the meridian of $e$ with color $(\alpha, 0)$. The cohomology class $\omega: H_1(M\backslash \Gamma; \mathbb{Z})\to G$ uniquely determines an element $\tilde{\omega}: H_1(M\backslash \tilde{\Gamma}; \mathbb{Z})\to G$ by requiring that $\tilde{\omega}$ sends the meridian of $m_e$ to $\alpha$. We want to show that by choosing $\alpha$ sufficiently generic, $L_k\cup \tilde{\Gamma}$ and $L_{k+l}\cup \tilde{\Gamma}$ can be connected by computable surgery presentations.

Suppose $L_{k+1}$ is obtained from $L_k$ by doing handle-slide move of a component/edge $e_1\subset L_k\cup \Gamma$ along a component $K_1$ of $L_k$. If $L_{k+1}$ is not computable, which means that the new $K_1$ in $L_{k+1}$ has multiplicity $\mathrm{mul}(e_1)\mathrm{mul}(K_1)=1\in G$, we do the following moves. We first slide $m_e$ along $K_1$, which changes the multiplicity of $K_1$ to $\alpha \mathrm{mul}(K_1)$, and then perform the handle-slide $s_1$, which further changes the multiplicity of $K_1$ to $\alpha\mathrm{mul}(e_1)\mathrm{mul}(K_1)$. We want to choose $\alpha$ so that both $\alpha \mathrm{mul}(K_1)$ and $\alpha\mathrm{mul}(e_1)\mathrm{mul}(K_1)$ are not $1$.

For the rest of moves $s_2, \cdots, s_l$, each time a component with multiplicity $1$ is created, we either consider a slide of $m_e$ as above along the component or reselect $\alpha$. We need to add more conditions to $\alpha$ so that all the handle-slide moves create computable surgery presentations. The conditions can be summarized as follows. 

{\bf Condition 1}: There is a finite set $\{x_i\}_{i\in I}\subset G$ and a finite set $J\subset {\mathbb{Z}}$ which only depends on the sequence \eqref{sequence}. We want to find an $\alpha$ so that $\alpha^{n} x_i\neq 1$ for all $i\in I$ and $n\in J$.
 
After the last handle-slide move, we get a computable presentation $L_{k+l}$ for $(M, \tilde{\Gamma}, \tilde{\omega})$. However, $m_e$ could be linked with $L_{k+l}$, and thus the multiplicities of $L_{k+l}$ might be different from the original multiplicities of $L_{k+l}$ in the sequence \eqref{sequence}. Since $m_e$ is isotopic to the meridian of the edge $e\subset \Gamma$, there is an isotopy of $m_e$ in $M$ which brings it back to the small meridian around $e$. This isotopy can be realized by a sequence of handle-slide moves
\begin{equation*}
\label{sequence2}
L_{k+l}\cup \tilde{\Gamma} \xrightarrow{h_1} L_{k+1}\cup \tilde{\Gamma}\xrightarrow{h_2} \cdots \xrightarrow{h_p} L_{k+l}\cup \tilde{\Gamma} \tag{H2}
\end{equation*}
The link $L_{k+l}$ after $h_p$ has the same multiplicities as the one in (\ref{sequence}), and we suppose the component $K_i$ of $L_{k+l}$ has multiplicity $\mathrm{mul}(K_i)$ after $h_p$. During each $h_j$, we slide $m_e$ along a component of $L_{k+l}$. So we see that the multiplicity of $K_i$ before the step $h_j$ has multiplicity of the form $\alpha^{m_{ij}}\mathrm{mul}(K_i)$ where $m_{ij}\in \mathbb{Z}$. In order to make all the surgery presentations in \eqref{sequence2} be computable, we add the condition that $\alpha^{m_{ij}}\mathrm{mul}(K_i)\neq 1$ for all $i, j$. Note that $m_{ij}$ only depends on the sequence \eqref{sequence2}. This condition can be summarized as follows.

{\bf Condition 2}: For a finite set $\{y_i\}_{i\in I'}\subset G$ and a finite set $J'\subset {\mathbb{Z}}$ which only depends on the sequence \eqref{sequence2}, we want to find an $\alpha$ so that $\alpha^{m}y_i\neq 1$ for $m\in J'$ and $i\in I'$. 

Since $G$ contains a $\mathbb{Z}$ summand, it is easy to see the existence of $\alpha\in G$ satisfying Conditions 1 and 2.

By Lemma \ref{handle-slide}, we have $\frac{\Delta (L_k\cup \tilde{\Gamma})}{2^{r(L_k)}(-1)^{\sigma_+(L_k)}}=\frac{\Delta (L_{k+l}\cup \tilde{\Gamma})}{2^{r(L_{k+l})}(-1)^{\sigma_+(L_{k+l})}}$. Let $\langle H\rangle=\alpha^{-2N}(\beta^2-\beta^{-2})$. By Example \ref{example}, we have $\Delta (L_k\cup \tilde{\Gamma})=\langle H\rangle \Delta (L_k\cup \Gamma)$ and $\Delta (L_{k+l}\cup \tilde{\Gamma})=\langle H\rangle \Delta (L_{k+l}\cup \Gamma)$. Therefore finally we have $\frac{\Delta (L_k\cup \Gamma)}{2^{r(L_k)}(-1)^{\sigma_+(L_k)}}=\frac{\Delta (L_{k+l}\cup \Gamma)}{2^{r(L_{k+l})}(-1)^{\sigma_+(L_{k+l})}}$.

Now we consider the case that $\Gamma=\emptyset$, which implies $L_k\neq \emptyset$ since $L_k\cup \Gamma\neq \emptyset$. Choose a component $K$ of $L_k$ with Kirby color $\Omega (\alpha, 1)$. We apply a positive and a negative blow-up of $K$ to create two new components $m_+$ and $m_{-}$, the Kirby colors of which are also $\Omega (\alpha, 1)$. The framing of $K$ is unchanged. Let $\tilde{\Gamma}=m_+\cup m_-$ and regard it as a graph in $M$. Then $\tilde{\omega}: H_1(M, \mathbb{Z}) \to G$ determines a cohomology class $\tilde{\omega}: H_1(M\backslash \tilde{\Gamma}, \mathbb{Z}) \to G$ which sends the meridians of $m_+$ and $m_-$ to $\alpha$.

Then $L_k$ is a computable presentation for $(M, \tilde{\Gamma}, \tilde{\omega})$. After performing the handle-slide moves in \eqref{sequence}, we get $L_{k+l}$, which is still a computable presentation for $(M, \tilde{\Gamma}, \tilde{\omega})$ since the handle-slide moves do not involve $m_+$ and $m_-$. Since $\tilde{\Gamma} \neq \emptyset$, as we proved above we have
$\frac{\Delta (L_k\cup \tilde{\Gamma})}{2^{r(L_k)}(-1)^{\sigma_+(L_k)}}=\frac{\Delta (L_{k+l}\cup \tilde{\Gamma})}{2^{r(L_{k+l})}(-1)^{\sigma_+(L_{k+l})}}$. On the other hand, by Lemma \ref{up-down}, we have $\Delta (L_k\cup \tilde{\Gamma})=-4\Delta (L_k)$. For $L_{k+l}$, $m_+$ and $m_-$ can be linked with several strands of $L_{k+l}$, as shown in the following figure.
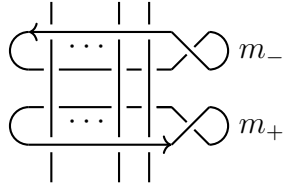
\begin{align*}
\begin{tikzpicture}[baseline=-0.65ex, thick, scale=1]
\draw (0,-0.4) to (0,0.9);
\draw (0,-0.6) to (0,-1);
\draw (0,1.1) to (0,1.4);
\draw (-0.4,-0.4) to (-0.4,0.9);
\draw (-0.4,-0.6) to (-0.4,-1);
\draw (-0.4,1.1) to (-0.4,1.4);
\draw (-1.3,-0.4) to (-1.3,0.9);
\draw (-1.3,-0.6) to (-1.3,-1);
\draw (-1.3,1.1) to (-1.3,1.4);
\draw (-0.8, 0.8) node {$\cdots$};
\draw (-0.8, -0.2) node {$\cdots$};
\draw (-1.6,1) [<-]to (0.3,1);
\draw (-1.6,0.5) to (-1.4,0.5);
\draw (-1.2,0.5) to (-0.5,0.5);
\draw (-0.3,0.5) to (-0.1,0.5);
\draw (0.1,0.5) to (0.3,0.5);
\draw  (-1.6,1) arc (90:270:0.25);
\draw (0.3,1) to (0.8,0.5);
\draw (0.3,0.5) to (0.5,0.7);
\draw (0.8,1) to (0.6,0.8);
\draw (-1.6,-0.5) [->]to (0.3,-0.5);
\draw (-1.6,0) to (-1.4,0);
\draw (-1.2,0) to (-0.5,0);
\draw (-0.3,0) to (-0.1,0);
\draw (0.1,0) to (0.3,0);
\draw (0.1,0) to (0.3,0);
\draw  (-1.6,0) arc (90:270:0.25);
\draw (0.3,0) to (0.5,-0.2);
\draw (0.6,-0.3) to (0.8,-0.5);
\draw (0.3,-0.5) to (0.8,0);
\draw (0.8,-0.5) arc (-90:90:0.25);
\draw (0.8,0.5) arc (-90:90:0.25);
\draw (1.5, -0.3) node {$m_{+}$};
\draw (1.5, 0.7) node {$m_{-}$};
\end{tikzpicture}
\end{align*}

To remove $m_+$ and $m_-$, we can first do several handle-slide moves along $m_+$ and $m_-$ to decrease the number of strands linked with $m_+$ and $m_-$, and then do one negative blow-up move, one positive blow-up move to remove $m_+$ and $m_-$. During this procedure, we can choose strands properly so that each step we get a computable presentation. By Lemmas \ref{up-down} and \ref{handle-slide}, we see that $\Delta (L_{k+l}\cup \tilde{\Gamma})=-4\Delta (L_{k+l})$ as well. As a result, we have $\frac{\Delta (L_k)}{2^{r(L_k)}(-1)^{\sigma_+(L_k)}}=\frac{\Delta (L_{k+l})}{2^{r(L_{k+l})}(-1)^{\sigma_+(L_{k+l})}}$.
\end{proof}

\section{Examples and calculations}
\subsection{A general formula for a class of lens spaces}
In this section, we compute $\Delta (L(mn-1, n), \omega):=\Delta (L(mn-1, n), \emptyset, \omega)$ for the lens space $L(mn-1, n)$ and $\Gamma=\emptyset$. We use the surgery presentation 
$L=\begin{tikzpicture}[baseline=-0.65ex, thick, scale=1]
\draw (-0.45,1) node {$\scriptstyle m$};
\draw (-0.3,1.1) to (-0.6,1.1);
\draw (-0.3,0.89) to (-0.3,1.11);
\draw (-0.3,0.90) to (-0.6,0.90);
\draw (-0.6,0.89) to (-0.6,1.11);
\draw (-0.3,1) [<-]to (0.3,0.4);
\draw (-0.3,0.4) to (-0.1,0.6);
\draw (0.1,0.8) [->] to (0.3,1);
\draw (-0.3,0.4) to (0.25,-0.09);
%
\draw (0.45,1) node {$\scriptstyle n$};
\draw (0.6,1.1) to (0.3,1.1);
\draw (0.6,0.89) to (0.6,1.11);
\draw (0.6,0.90) to (0.3,0.90);
\draw (0.3,0.89) to (0.3,1.11);
\draw (0.1,0.2) to (0.3,0.4);
\draw  (-0.6,1) arc (90:315:0.6);
\draw  (0.6,1) arc (90:-140:0.6);
\end{tikzpicture}$ for $L(mn-1, n)$, where $m>0$ or $n>0$ inside a square represents the number of positive full twists.
For a cohomology class
\begin{align}\label{eq:Lens}
\omega : H_1(M, \mathbb{Z})=\left < [m_1], [m_2]  \middle| \left(\begin{matrix} m & -1 \\ -1 & n \end{matrix} \right) \left(\begin{matrix} [m_1] \\ [m_2] \end{matrix} \right) =0 \right > \to G,
\end{align}
where $m_1$ (resp. $m_2$) is the meridian of the left-hand (resp. right-hand) slide component of $L$.
Let $u=\omega([m_1])$ and $v=\omega([m_2])$.
In the following calculations, a diagram inside round brackets represents the Alexander polynomial of the diagram. We have
\begin{align*}
&\Delta (L(mn-1, n), \omega)= \frac{\Delta (L)}{2^{r}(-1)^{\sigma_+ (L)}} \\
&= \frac{1}{2^{2}(-1)^{\sigma_+(L)}} \left(
\begin{tikzpicture}[baseline=-0.65ex, thick, scale=1]
\draw (-0.45,1) node {$\scriptstyle m$};
\draw (-0.3,1.1) to (-0.6,1.1);
\draw (-0.3,0.89) to (-0.3,1.11);
\draw (-0.3,0.90) to (-0.6,0.90);
\draw (-0.6,0.89) to (-0.6,1.11);
\draw (-0.3,1) [<-]to (0.3,0.4);
\draw (-0.3,0.4) to (-0.1,0.6);
\draw (0.1,0.8) [->] to (0.3,1);
\draw (-0.3,0.4) to (0.25,-0.09);
\draw (0.45,1) node {$\scriptstyle n$};
\draw (0.6,1.1) to (0.3,1.1);
\draw (0.6,0.89) to (0.6,1.11);
\draw (0.6,0.90) to (0.3,0.90);
\draw (0.3,0.89) to (0.3,1.11);
\draw (0.1,0.2) to (0.3,0.4);
\draw  (-0.6,1) arc (90:315:0.6);
\draw  (0.6,1) arc (90:-140:0.6);
\draw (0.5,-0.5) node {$\scriptstyle \Omega(v, 1)$};
\draw (-0.5,-0.5) node {$\scriptstyle \Omega(u, 1)$};
\end{tikzpicture}
\right)\\
&= \frac{d(u)d(v)}{4(-1)^{\sigma_+(L)}} \left(
\begin{tikzpicture}[baseline=-0.01ex, thick, scale=1]
\draw (-0.45,1) node {$\scriptstyle m$};
\draw (-0.3,1.1) to (-0.6,1.1);
\draw (-0.3,0.89) to (-0.3,1.11);
\draw (-0.3,0.90) to (-0.6,0.90);
\draw (-0.6,0.89) to (-0.6,1.11);
\draw (-0.3,1) [<-]to (0.3,0.4);
\draw (-0.3,0.4) to (-0.1,0.6);
\draw (0.1,0.8) [->] to (0.3,1);
\draw (-0.3,0.4) to (0.25,-0.09);
\draw (0.45,1) node {$\scriptstyle n$};
\draw (0.6,1.1) to (0.3,1.1);
\draw (0.6,0.89) to (0.6,1.11);
\draw (0.6,0.90) to (0.3,0.90);
\draw (0.3,0.89) to (0.3,1.11);
\draw (0.1,0.2) to (0.3,0.4);
\draw  (-0.6,1) arc (90:315:0.6);
\draw  (0.6,1) arc (90:-140:0.6);
\draw (0.5,-0.5) node {$\scriptstyle (v, 1)$};
\draw (-0.5,-0.5) node {$\scriptstyle (u, 1)$};
\end{tikzpicture}
-
\begin{tikzpicture}[baseline=-0.01ex, thick, scale=1]
\draw (-0.45,1) node {$\scriptstyle m$};
\draw (-0.3,1.1) to (-0.6,1.1);
\draw (-0.3,0.89) to (-0.3,1.11);
\draw (-0.3,0.90) to (-0.6,0.90);
\draw (-0.6,0.89) to (-0.6,1.11);
\draw (-0.3,1) [->]to (0.3,0.4);
\draw (-0.3,0.4) to (-0.1,0.6);
\draw (0.1,0.8) [->] to (0.3,1);
\draw (-0.3,0.4) to (0.25,-0.09);
\draw (0.45,1) node {$\scriptstyle n$};
\draw (0.6,1.1) to (0.3,1.1);
\draw (0.6,0.89) to (0.6,1.11);
\draw (0.6,0.90) to (0.3,0.90);
\draw (0.3,0.89) to (0.3,1.11);
\draw (0.1,0.2) to (0.3,0.4);
\draw  (-0.6,1) arc (90:315:0.6);
\draw  (0.6,1) arc (90:-140:0.6);
\draw (0.5,-0.5) node {$\scriptstyle (v, 1)$};
\draw (-0.5,-0.5) node {$\scriptstyle (u^{-1}, 1)$};
\end{tikzpicture}
-
\begin{tikzpicture}[baseline=-0.01ex, thick, scale=1]
\draw (-0.45,1) node {$\scriptstyle m$};
\draw (-0.3,1.1) to (-0.6,1.1);
\draw (-0.3,0.89) to (-0.3,1.11);
\draw (-0.3,0.90) to (-0.6,0.90);
\draw (-0.6,0.89) to (-0.6,1.11);
\draw (-0.3,1) [<-]to (0.3,0.4);
\draw (-0.3,0.4) to (-0.1,0.6);
\draw (0.1,0.8) [<-] to (0.3,1);
\draw (-0.3,0.4) to (0.25,-0.09);
\draw (0.45,1) node {$\scriptstyle n$};
\draw (0.6,1.1) to (0.3,1.1);
\draw (0.6,0.89) to (0.6,1.11);
\draw (0.6,0.90) to (0.3,0.90);
\draw (0.3,0.89) to (0.3,1.11);
\draw (0.1,0.2) to (0.3,0.4);
\draw  (-0.6,1) arc (90:315:0.6);
\draw  (0.6,1) arc (90:-140:0.6);
\draw (0.5,-0.5) node {$\scriptstyle (v^{-1}, 1)$};
\draw (-0.5,-0.5) node {$\scriptstyle (u, 1)$};
\end{tikzpicture}
+
\begin{tikzpicture}[baseline=-0.01ex, thick, scale=1]
\draw (-0.45,1) node {$\scriptstyle m$};
\draw (-0.3,1.1) to (-0.6,1.1);
\draw (-0.3,0.89) to (-0.3,1.11);
\draw (-0.3,0.90) to (-0.6,0.90);
\draw (-0.6,0.89) to (-0.6,1.11);
\draw (-0.3,1) [->]to (0.3,0.4);
\draw (-0.3,0.4) to (-0.1,0.6);
\draw (0.1,0.8) [<-] to (0.3,1);
\draw (-0.3,0.4) to (0.25,-0.09);
\draw (0.45,1) node {$\scriptstyle n$};
\draw (0.6,1.1) to (0.3,1.1);
\draw (0.6,0.89) to (0.6,1.11);
\draw (0.6,0.90) to (0.3,0.90);
\draw (0.3,0.89) to (0.3,1.11);
\draw (0.1,0.2) to (0.3,0.4);
\draw  (-0.6,1) arc (90:315:0.6);
\draw  (0.6,1) arc (90:-140:0.6);
\draw (0.5,-0.5) node {$\scriptstyle (v^{-1}, 1)$};
\draw (-0.5,-0.5) node {$\scriptstyle (u^{-1}, 1)$};
\end{tikzpicture}
\right)\\
&= \frac{d(u)d(v)}{4(-1)^{\sigma_+(L)}} {\Big[} u^{-2m} v^{-2n} 
\left(\begin{tikzpicture}[baseline=+1ex, thick, scale=1]
\draw (-0.3,1) [<-]to (0.3,0.4);
\draw (-0.3,0.4) to (-0.1,0.6);
\draw (0.1,0.8) [->] to (0.3,1);
\draw (-0.3,0.4) to (0.3,-0.2);
\draw (-0.3,-0.2) to (-0.1,0);
\draw (0.1,0.2) to (0.3,0.4);
\draw  (-0.3,1) arc (90:270:0.6);
\draw  (0.3,1) arc (90:-90:0.6);
\draw (0.5,-0.5) node {$\scriptstyle (v, 1)$};
\draw (-0.5,-0.5) node {$\scriptstyle (u, 1)$};
\end{tikzpicture}\right) 
- u^{2m} v^{-2n} 
\left(\begin{tikzpicture}[baseline=1ex, thick, scale=1]
\draw (-0.3,1) [->]to (0.3,0.4);
\draw (-0.3,0.4) to (-0.1,0.6);
\draw (0.1,0.8) [->] to (0.3,1);
\draw (-0.3,0.4) to (0.3,-0.2);
\draw (-0.3,-0.2) to (-0.1,0);
\draw (0.1,0.2) to (0.3,0.4);
\draw  (-0.3,1) arc (90:270:0.6);
\draw  (0.3,1) arc (90:-90:0.6);
\draw (0.5,-0.5) node {$\scriptstyle (v, 1)$};
\draw (-0.5,-0.5) node {$\scriptstyle (u^{-1}, 1)$};
\end{tikzpicture}\right) \\
&\qquad\qquad\qquad\quad
-u^{-2m} v^{2n}
\left(\begin{tikzpicture}[baseline=1ex, thick, scale=1]
\draw (-0.3,1) [<-]to (0.3,0.4);
\draw (-0.3,0.4) to (-0.1,0.6);
\draw (0.1,0.8) [<-] to (0.3,1);
\draw (-0.3,0.4) to (0.3,-0.2);
\draw (-0.3,-0.2) to (-0.1,0);
\draw (0.1,0.2) to (0.3,0.4);
\draw  (-0.3,1) arc (90:270:0.6);
\draw  (0.3,1) arc (90:-90:0.6);
\draw (0.5,-0.5) node {$\scriptstyle (v^{-1}, 1)$};
\draw (-0.5,-0.5) node {$\scriptstyle (u, 1)$};
\end{tikzpicture}\right)
+ u^{2m} v^{2n}
\left(\begin{tikzpicture}[baseline=1ex, thick, scale=1]
\draw (-0.3,1) [->]to (0.3,0.4);
\draw (-0.3,0.4) to (-0.1,0.6);
\draw (0.1,0.8) [<-] to (0.3,1);
\draw (-0.3,0.4) to (0.3,-0.2);
\draw (-0.3,-0.2) to (-0.1,0);
\draw (0.1,0.2) to (0.3,0.4);
\draw  (-0.3,1) arc (90:270:0.6);
\draw  (0.3,1) arc (90:-90:0.6);
\draw (0.5,-0.5) node {$\scriptstyle (v^{-1}, 1)$};
\draw (-0.5,-0.5) node {$\scriptstyle (u^{-1}, 1)$};
\end{tikzpicture}\right){\Big]}\\
&= -\frac{d(u)d(v)}{4(-1)^{\sigma_+(L)}} \left( u^{2-2m} v^{2-2n} 
+ u^{2+2m} v^{-2-2n}  
+u^{-2-2m} v^{2+2n}
+ u^{-2+2m} v^{-2+2n}\right)\\
&= -\frac{d(u)d(v)}{4(-1)^{\sigma_+(L)}} (u^2 v^{-2n} + u^{-2} v^{2n})(u^{2m} v^{-2} + u^{-2m} v^2),  
\end{align*}
where the third equality follows from the definition of Kirby color, the forth one is because a positive full-twist contribute $t^{-2N}$ if the strand has color $(t, N)$, and the fifth one follows from Example~\ref{example}.

Here note that (\ref{eq:Lens}) implies $u^m v^{-1} =1$, $u^{-1} v^n =1$.  Thus
\begin{align*}
\Delta(L(mn-1, n), \omega)&= -\frac{d(u)d(v)}{4(-1)^{\sigma_+(L)}} (u^2 v^{-2n} + u^{-2} v^{2n})(u^{2m} v^{-2} + u^{-2m} v^2) \\
&= -\frac{d(u)d(v)}{4(-1)^{\sigma_+(L)}} ((u v^{-n})^2 + (u^{-1} v^{n})^2)((u^{m} v^{-1})^{2} + (u^{-m} v)^2)\\
&= (-1)^{\sigma_+(L)+1} d(u)d(v).  
\end{align*}
Then we have
\begin{prop}
\label{prop51}
\[\Delta(L(mn-1, n), \omega) = (-1)^{\sigma_+(L)+1} d(u)d(v).\]
\end{prop}

\subsection{$L(7, 1)$ and $L(7, 2)$}
It is known that lens spaces $L(7, 1)$ and $L(7, 2)$ are homotopy equivalent but not homeomorphic. We show that our invariant can distinguish them.

Let $\xi=\mathrm{exp}(\frac{2\pi i}{7})$, $B=\mathbb{Q}(\pi, \xi)$ the extension field of $\mathbb{Q}$ generated by $\pi$ and $\xi$, and $G=\mathbb{Z}\langle \pi, \xi \rangle$ the abelian group generated by $\pi$ and $\xi$. We consider $\Delta (L(7, 1), \omega)$ and $\Delta (L(7, 2), \omega)$ for this $1$-palette $(B, G)$.

\begin{prop}\label{Prop5152}
The invariant $\Delta (M, \omega)$ corresponding to the $1$-palette $(B, G)$ where $B=\mathbb{Q}(\pi, \xi)$ and $G=\mathbb{Z}\langle \pi, \xi \rangle$ distinguishes $L(7, 1)$ and $L(7, 2)$.  
More concretely, there exists a cohomology class $\omega_0$ for $L(7, 1)$ such that for any cohomology class $\omega$ for $L(7, 2)$, we have 
\[
\Delta (L(7, 1), \omega_0)  \neq \Delta (L(7, 2), \omega).  
\]
\end{prop}
\begin{proof}
Note that $L(7, 1)=L(mn-1, n)$ for $m=8, n=1$, and $L(7, 2)=L(mn-1, n)$ for $m=4, n=2$. So we can apply the discussion we did in Section 5.1.
A cohomology class $$\omega: H_1(L(7, 2), \mathbb{Z})\cong \mathbb{Z}/7\mathbb{Z}\to \mathbb{Z}\langle \pi, \xi \rangle$$ is determined by $\omega([m_1])$ and $\omega([m_1])$, which satisfy 
\[
\left(\begin{matrix} 4 & -1  \\ -1 & 2 \end{matrix}\right)  \left(\begin{matrix} \omega[m_1] \\ \omega[m_2] \end{matrix}\right) = \left(\begin{matrix} 1 \\ 1 \end{matrix}\right).  
\] 
So we have totally six non-trivial cohomology classes which are given by
\[  \omega_1 : \left(\begin{matrix} \xi^2 \\ \xi \end{matrix}\right), \,\,
\omega_2 :\left(\begin{matrix} \xi^4 \\ \xi^2 \end{matrix}\right), \,\,\omega_3 :
\left(\begin{matrix} \xi^6 \\ \xi^3 \end{matrix}\right), \,\, \omega_4 :
\left(\begin{matrix} \xi \\ \xi^4 \end{matrix}\right),  \,\, \omega_5 :
\left(\begin{matrix} \xi^3 \\ \xi^5 \end{matrix}\right),\,\, \omega_6 :
\left(\begin{matrix} \xi^5 \\ \xi^6 \end{matrix}\right).     
\] 
Let $u_i=\omega_i([m_1])$ and $v_i=\omega_i([m_2])$. By Prop. \ref{prop51} we have $$\Delta(L(7, 2), \omega_i)=-d(u_i)d(v_i).$$
 
Similarly we can consider the non-trivial cohomology classes for $L(7, 1)$. We see that $\omega_0:\begin{pmatrix} \xi \\ \xi \end{pmatrix}$ is one of them. The corresponding invariant is
$$\Delta(L(7, 1), \omega_0)=-d(\xi)d(\xi).$$

We claim that $\Delta(L(7, 2), \omega_i)\neq \Delta(L(7, 1), \omega_0)$ for $1\leq i \leq 6$, which can be confirmed by directly calculations. For instance $\Delta(L(7, 2), \omega_1)=\Delta(L(7, 1), \omega_0)\iff d(\xi^2)=d(\xi)\iff \xi^4-\xi^{-4}=\xi^2-\xi^{-2}\iff \xi^2+\xi^{-2}=1$, which is impossible since the minimal polynomial of $\xi$ is $\sum_{k=0}^6\xi^k=0$.
\end{proof}

\bibliographystyle{siam}
\bibliography{bao}

\end{document}